\documentclass[11pt,reqno]{amsart}

\usepackage{color, tikz}
\usepackage{amsmath, amssymb, mathtools, array, ytableau, graphicx, subcaption}
\usepackage{pstricks,pst-node}
\usepackage{tikz,varwidth}
\usepackage{float}

\usepackage[top=1.2in, bottom=1in, left=.9in, right=.9in]{geometry}
\newtheorem{theorem}{Theorem}[section]
\newtheorem{lemma}[theorem]{Lemma}

\newtheorem{proposition}[theorem]{Proposition}

\newtheorem{conjecture}[theorem]{Conjecture}

\theoremstyle{remark}
\newtheorem*{remark}{Remark}

\numberwithin{equation}{section}

\def\Log{\operatorname{Log}}
\def\Li{\operatorname{Li}}
\renewcommand{\Re}{{\rm Re}}

\makeatletter
\@namedef{subjclassname@2020}{%
  \textup{2020} Mathematics Subject Classification}
\makeatother

\renewcommand{\Re}{{\rm Re}}

\def\Log{\operatorname{Log}}
\def\Li{\operatorname{Li}}

\begin{document}

%%%%%%%%%%%%%%%%%%%%%%%%%%%%%%%%%%%%%%%%%
%               Title, Etc.             %
%%%%%%%%%%%%%%%%%%%%%%%%%%%%%%%%%%%%%%%%%
\title[Inequalities for $t$-hook numbers]{Inequalities for the number of $t$-hooks in two partition classes arising from sum-product identities}

\date{\today}

\author[A. Dhar]{Aritram Dhar}
\address{Department of Mathematics, University of Florida, Gainesville, FL 32601, USA}
\email{aritramdhar@ufl.edu}

\author[B. Kim]{Byungchan Kim}
\address{School of Natural Sciences, Seoul National University of Science and Technology, 232 Gongneung-ro, Nowon-gu, Seoul, 01811, Republic of Korea}
\email{bkim4@seoultech.ac.kr}

\author[E. Kim]{Eunmi Kim$^\ast$}
\address{Institute of Mathematical Sciences, Ewha Womans University, 52 Ewhayeodae-gil, Seodaemun-gu, Seoul 03760, Republic of Korea}
\email{ekim67@ewha.ac.kr; eunmi.kim67@gmail.com}
\thanks{$^\ast$ Corresponding author.}

\author[A.J. Yee]{Ae Ja Yee}
\address{Department of Mathematics, The Pennsylvania State University, University Park, PA 16802, USA}
\email{yee@psu.edu}

%%%%%%%%%%%%%%%%%%%%%%%%%%%%%%%
%          Abstract           %
%%%%%%%%%%%%%%%%%%%%%%%%%%%%%%%
\begin{abstract} 
Motivated by recent study on the number of $t$-hooks in partitions arising from Euler's partition identity, we investigate the number of $t$-hooks in the sets from the first Rogers-Ramanujan identity and the first little G\"ollitz identity. In particular, for $t=1,2$, we obtain the generating functions for the number of $t$-hooks and prove $t$-hook inequalities by deriving asymptotic formulas.
\end{abstract}

%%%%%%%%%%%%%%%%%%%%%%%%%%%%%%%%%%%%%%%%%
%               Document Text           %
%%%%%%%%%%%%%%%%%%%%%%%%%%%%%%%%%%%%%%%%%
\maketitle

%%%%%%%%%%%%%%%%%%%%%%%%%%%%%%%%%%%%%%%%%%%%%
%          Section. Introduction      %
%%%%%%%%%%%%%%%%%%%%%%%%%%%%%%%%%%%%%%%%%%%%%
\section{Introduction}
A \textit{partition} $\lambda$ is a weakly decreasing finite sequence $\lambda = (\lambda_1,\lambda_2,\dots,\lambda_k)$ of positive integers. The elements $\lambda_i$ appearing in the sequence $\lambda$ are called the \textit{parts} of $\lambda$. The sum of all the parts of $\lambda$ is called the \textit{size} of $\lambda$ and denoted by $|\lambda|$. We say $\lambda$ is a partition of $n$ if its size is equal to $n$.

The most foundational and simplest identity in the theory of partitions due to Euler states that the number of partitions of $n$ into distinct parts is equal to the number of partitions of $n$ into odd parts \cite[Corollary 1.2]{andrews_book}. This identity has inspired extensive research over the past several centuries leading to various generalizations and refinements, including the work of Glaisher \cite{glaisher} and Sylvester \cite{sylvester} from the late 19th century, the generalization of Andrews \cite{andrews1969}, and the recent contribution of Andrews, Kumar and Yee \cite{AKY}. Recently, Andrews studied Euler's identity from a different angle while confirming two conjectures of Beck \cite{andrews2017, beck1}.  One of the results of Andrews can be reformulated as follows \cite[Theorem 2]{andrews2017}.
\begin{quote}
The total number of parts in the partitions of $n$ into distinct parts is greater than or equal to the total number of different parts in the partitions of $n$ into odd parts. 
\end{quote}

This inequality result was generalized by Ballantine, Burson, Craig, Folsom, and Wen to inequalities involving partition hooks \cite{ballantine1}.

The \textit{Young diagram} of a partition $\lambda$ is a way of representing $\lambda$ graphically where the parts $\lambda_i$ of $\lambda$, called \textit{cells}, are arranged in left-justified rows with $\lambda_i$ cells in the $i$-th row. The conjugate of a partition $\lambda = (\lambda_1,\lambda_2,\dots,\lambda_k)$ is the partition $\lambda^{\prime} = (\lambda^{\prime}_1,\lambda^{\prime}_2,\dots,\lambda^{\prime}_{\lambda_1})$ whose Young diagram has the columns of $\lambda$ as rows. For a cell in the $i$-th row and $j$-th column of the Young diagram of a partition $\lambda$, its \textit{hook length} is defined as $h(i,j) = \lambda_i + \lambda^{\prime}_{j} - i - j + 1$. See Figure \ref{fig1} below for an example of a Young diagram of a partition along with its hook lengths. A hook of length $t$ is called a $t$-hook. 
\begin{figure}[H]
\ytableausetup{centertableaux}
\ytableaushort
{{11}{9}{6}{5}{3}{2}{1},{7}{5}{2}{1},{4}{2},{3}{1},{1}}
* {7,4,2,2,1}
\par\quad\\
\caption{The Young diagram of the partition $\lambda = (7,4,2,2,1)$ with its hook lengths.}
\label{fig1}
\end{figure}

%We say that a partition has a \textit{$t$-hook} if there exists a hook of length $t$ in the Young diagram of the partition.

Let $a_t(n)$ be the total number of  $t$-hooks  in all partitions of $n$ into odd parts and $b_t(n)$ be the total number of $t$-hooks in all partitions of $n$  into distinct parts. The result of Andrews from \cite[Theorem 2]{andrews2017} implies 
\begin{equation}\label{andrews_ineq}
a_1(n) \ge b_1(n) \text{ for $n\ge 1$}.
\end{equation}
Noting this identity 
\[
\sum_{t\ge 1} a_t(n) =\sum_{t\ge 1} b_t(n) \text{ for $n\ge 1$ },
\]
Ballantine, Burson, Craig, Folsom, and Wen conjectured that the inequality of Andrews in \eqref{andrews_ineq} reverses for $t>1$ when $n$ is large enough \cite{ballantine1}. 

\begin{conjecture}[Ballantine--Burson--Craig--Folsom--Wen \cite{ballantine1}] \label{conj1.1}
For every integer $t\ge 2$, 
\begin{equation}\label{conj1.1_ineq}
a_t(n) - b_t(n) \to \infty \text{ as } n \to \infty. 
\end{equation} 
\end{conjecture}

The $t=2,3$ cases of this conjecture were proved by the authors of the paper \cite{ballantine1} and the remaining cases were settled by Craig, Dawsey, and Han \cite{craig1}. 

Euler's identity connects partitions with a gap condition (distinct parts) and those with a part congruence condition (odd parts), and many other partition identities share this structure. The main objective of this paper is to extend the hook bias question in Conjecture~\ref{conj1.1} to identities of this type,  more specifically, we will have our focus on identities involving gap $2$ conditions. 

The celebrated Rogers–Ramanujan identities, which were originally proved by Rogers \cite{rogers} and rediscovered by Ramanujan \cite{ramanujan}, are perfect examples of such identities, relating gap $2$ conditions on parts to congruence conditions modulo $5$.  The first Rogers--Ramanujan identity states that the number of partitions of $n$ with parts differing by $2$ is equal to the number of partitions of $n$ into parts congruent to $1$ or $4$ mod $5$. 

\iffalse
The following is the first Rogers--Ramanujan identity:
\begin{align*}
	\sum_{n\ge 0} \frac{ q^{n^2}}{(q;q)_n} =\frac{1}{(q,q^4;q^{5})_{\infty}}.
\end{align*}   

Let 
\begin{align*}
	\mathcal{R}_1&:=\{\lambda \, |\, \lambda_{i}-\lambda_{i+1} \ge 2 \},\\
	\mathcal{R}_2&:=\{\lambda \, |\, \lambda_{i}\equiv 1,4 \pmod{5}\}.%,\\
%	\mathcal{R}_3&:=\{\lambda \, |\, \text{no parts below the Durfee square i.e., }  \lambda_{\ell} \ge \ell \}.
\end{align*} 
For $j=1,2$, let 
\begin{align*}
	\widetilde{r}_{j, t} (\lambda)&:= \# \text{ $t$-hooks in $\lambda \in \mathcal{R}_j$},
\end{align*}
and
\begin{align*}
	r_{j,t}(n)&:=\sum_{|\lambda|=n} \widetilde{r}_{j,t} (\lambda).
\end{align*}
\fi

For a positive integer $t$, let $r_{1,t}(n)$ be the number of $t$-hooks in all partitions of $n$ with gap between parts being at least $2$ and and $r_{2,t}(n)$ be the number of $t$-hooks in all partitions of $n$ with parts congruent to $1,4$ modulo $5$.  Our first main result is the following inequalities, which are comparable to the inequalities in \eqref{andrews_ineq} and \eqref{conj1.1_ineq}.

\begin{theorem}\label{thm:r_ineq}
	For sufficiently large $n$, 
	\[
		r_{1,1} (n) > r_{2,1}(n) \quad\text{and}\quad r_{1,2} (n) < r_{2,2} (n). 
	\]
\end{theorem}

\iffalse
The following is the first little G\"ollnitz identity:
\begin{align*}
	\sum_{n\ge 0} \frac{ q^{n^2+n} (-q^{-1};q^2)_n}{(q^2;q^2)_n}=(-q,-q^2, -q^4;q^4)_{\infty}  =\frac{1}{(q,q^5,q^6;q^{8})_{\infty}}.
\end{align*}   
Let 
\begin{align*}
	\mathcal{G}_1&:=\{\lambda \, |\, \lambda_{i}-\lambda_{i+1} \ge 2, \;\; \lambda_{i}-\lambda_{i+1} >2 \text{ if $\lambda_i\equiv 0 \pmod{2}$}\},\\
%	\mathcal{G}_2&:=\{\lambda \, |\, \lambda_{i}-\lambda_{i+1}\ge 1, \;\; \lambda_{i} \equiv 0, 1, 2 \pmod{4} \},\\
	\mathcal{G}_2&:=\{\lambda \, |\, \lambda_{i}\equiv 1,5, 6 \pmod{8}\}.
\end{align*} 
For $j=1,2$, let 
\begin{align*}
	\widetilde{g}_{j,t} (\lambda) &:= \# \text{ $t$-hooks in $\lambda \in \mathcal{G}_j$},
\end{align*}
and
\begin{align*}
	g_{j,t}(n)&=\sum_{|\lambda|=n} \widetilde{g}_{j,t} (\lambda).
\end{align*}
\fi

We prove these inequalities by deriving asymptotic formulas for $r_{1,1}(n), r_{1,2}(n), r_{2,1}(n)$ and $r_{2,2}(n)$, which are given in the following theorems.

\begin{theorem}\label{thm:r1_asymp}
	As $n \to \infty$,
	\[
		r_{1,1}(n) \sim r_{1,2}(n) \sim \frac{3^{\frac14} \phi^{\frac12}  \log(\phi) }{2 \pi} n^{-\frac14}e^{2\pi \sqrt{\frac{n}{15}}},
	\]
	where $\phi$ is the golden ratio. 
\end{theorem}

 \begin{theorem}\label{thm:r2_asymp}
	As $n \to \infty$,
	\[
		r_{2,1}(n) \sim  \frac{3^{\frac14}  \phi^{\frac12}}{5 \pi} n^{-\frac14}e^{2\pi \sqrt{\frac{n}{15}}} \qquad \text{and} \qquad r_{2,2}(n) \sim  \frac{3^{\frac54}  \phi^{\frac12}}{10 \pi} n^{-\frac14}e^{2\pi \sqrt{\frac{n}{15}}}. 
	\]
\end{theorem}

The generating functions for $r_{1,1}(n), r_{1,2}(n), r_{2,1}(n)$ and $r_{2,2}(n)$ are either a form of a Nahm sum or a product of an infinite product and a rational function. Here by a Nahm sum, we mean that a $q$-series $\sum_{n \geq 0} A_n (q)$ with $A_{n+1}(q)/A_{n}(q)$ is a rational function in $q$. Nahm sums have played an important role in partition theory, $q$-series, and modular forms \cite{BMRS, VZ, W, Z}. We employ Ingham's tauberian theorem (see Section \ref{sec2}) as these $t$-hook numbers are weakly increasing and this method only requires a detailed asymptotic behavior near $q=1$. Nonetheless, investigating the asymptotic behavior of Nahm sums requires several non-trivial estimates. Our idea of proofs is similar to that of Bringmann,  Man, Rolen, and Storzer \cite{BMRS}, but we simplify their proofs accordingly and give more necessary details.  More details on how we obtain the necessary asymptotics are given in Section \ref{sec4}.

%Our second example is the little G\"ollnitz identities whose gap conditions are a bit different from those of the G\"ollnitz--Gordon identities, namely parts differ by at least $2$ but no odd parts differing by exactly $2$ are allowed. The first little G\"ollnitz identity states the number of partitions of $n$ with parts differing by $2$ and no odd parts differing by exactly $2$ is equal to the number of partitions of $n$ into parts congruent to $1, 5$ or $6$ modulo $8$. 

Our second example is the little G\"ollnitz identities, which are analogues of the Rogers--Ramanujan identities for modulo 8  \cite{gollnitz}. 
%whose gap conditions are a bit different from those of the G\"ollnitz--Gordon identities, namely parts differ by at least $2$ but no odd parts differing by exactly $2$ are allowed. 
The first little G\"ollnitz identity states that the number of partitions of $n$ with parts differing by $2$ and no odd parts differing by exactly $2$ is equal to the number of partitions of $n$ into parts congruent to $1, 5$ or $6$ modulo $8$.

For a positive integer $t$, let $g_{1,t}(n)$ be the number of $t$-hooks in all partitions of $n$ with gap between parts being at least $2$ while no odd parts differing by exactly $2$ are allowed. We also let $g_{2,t}(n)$ be the number of $t$-hooks in all partitions of $n$ with parts congruent to $1,5, 6$ modulo $8$.  Our second result is the following inequalities. 

\begin{theorem}\label{thm:g_ineq}
	For sufficiently large $n$,
	\[
		g_{1,1}(n) > g_{2,1}(n) \quad \text{and}\quad  g_{1,2}(n) < g_{2,2}(n).
	\]
\end{theorem}
They follow from the asymptotic formulas below. 
\begin{theorem}\label{thm:g1_asymp}
	As $n \to \infty$,
	\[
		g_{1,1}(n) \sim g_{1,2}(n) \sim \frac{\log(\sqrt{2}+1) }{2^{\frac54}\pi } n^{-\frac14} e^{\frac{\pi}{2} \sqrt{n}}.
	\]
\end{theorem}

\begin{theorem}\label{thm:g2_asymp}
	As $n \to \infty$,
	\[
		g_{2,1}(n) \sim \frac{3}{2^{\frac{13}{4}}\pi} n^{-\frac14} e^{\frac{\pi}{2} \sqrt{n}} \quad \text{and} \quad g_{2,2}(n) \sim  \frac{1}{2^{\frac{5}{4}}\pi} n^{-\frac14} e^{\frac{\pi}{2} \sqrt{n}} .
	\]
\end{theorem}

The corresponding generating functions are also either Nahm sums or a product of an infinite product and a rational function. As the generating function is more complicated, there are additional technical difficulties. However, the main idea of the proofs remains the same as before.

The rest of this paper is organized as follows. In Section~\ref{sec2}, we recall some background and necessary results from the literature. In Sections~\ref{sec3} and \ref{sec4}, we give generating functions for the first Rogers--Ramanujan partitions and their asymptotics, respectively.  The first little G\"ollnitz identity is studied in Sections~\ref{sec5} and \ref{sec6}. We then offer some  comments in the final section.

%%%%%%%%%%%%%%%%%%%%%%%%%%%%%%%%%%%%%%%%%%%%%
%          Section.   %
%%%%%%%%%%%%%%%%%%%%%%%%%%%%%%%%%%%%%%%%%%%%%
\section{Preliminaries}\label{sec2}

%We now borrow some notations from Andrews' encyclopedia \cite{andrews_book}. 
For complex variables $a$, $a_1,a_2,\ldots,a_k$, $q$ and $n \in \mathbb{Z}_{\ge 0} \cup \{\infty\}$, the conventional $q$-Pochhammer symbol is defined as
\begin{align*}
(a;q)_n := \prod\limits_{i=0}^{n-1}(1-aq^i)\quad\text{and}\quad (a_1,a_2,\ldots,a_k;q)_n := \prod\limits_{i=1}^{k}(a_i;q)_n.
\end{align*}
Here and throughout the paper, we assume $|q|<1$.

We need to derive an asymptotic formula for $1/(q;q)_n$ and this can be achieved using two lemmas %Lemmas 2.1 and 2.2 
in \cite[Lemmas 2.1, 2.2]{BMRS}. In particular, Lemma 2.2 in \cite{BMRS} is a complex variable extension of Zagier's previous result \cite[Lemma 2.1]{GZ} for $\zeta=1$. 

Throughout this paper, we set $q=e^{-z}$ where $z = \varepsilon (1+iy)$ with $\varepsilon >0$ and $y \in \mathbb{R}$.

\begin{lemma}[{\cite[Lemma 2.2]{BMRS}}]\label{lem:Zagier_asym}
	Let $z, w \in \mathbb{C}$ with ${\rm Re}(z) >0$, $|w|<1$, and $\nu \in \mathbb{C}$ with $\nu z = o(1)$. Then
	\[
		\Log \left( (we^{-\nu z} q;q)_\infty \right) = - \Li_2 (w) \frac{1}{z} - \left( \nu + \frac{1}{2} \right) \Log(1-w) - \frac{\nu^2 z}{2} \frac{w}{1-w} + \psi_w(\nu, z),
	\]
	where $\psi_w (\nu, z)$, for $R \in \mathbb{N}$, has an asymptotic expansion as $z\to 0$ with $\Re(z)>0$
	\begin{equation}\label{eq:psi}
		\psi_w(\nu, z) = -\sum_{r=2}^{R-1} \left( B_r(-\nu)-\delta_{r,2} \nu^2 \right) \Li_{2-r}(w) \frac{z^{r-1}}{r!} + O\left( z^{R-1} \right)
	\end{equation}
	with $B_r(x)$ the Bernoulli polynomial and $\delta_{i,j}$ the Kronecker delta symbol. In particular, for every $n \in \mathbb{N}_0$, the coefficient of $\nu^n$ is $O\left(z^{\frac{2n}{3}} \right)$.
\end{lemma}

\begin{lemma}[{\cite[Lemma 2.1]{BMRS}}]\label{lem:eta_asym}
Suppose that $y \ll \varepsilon^{-\frac{1}{2} + \delta}$ for some $\delta>0$. Then we have that as $z \to 0$, 
\[
	- \Log \left( (q;q)_\infty \right) = \frac{\pi^{2}}{6z} + \frac{1}{2} \Log\left( \frac{z}{2\pi} \right) - \frac{z}{24} + O(z^L)
\]
for any $L \in \mathbb{N}$. 
\end{lemma}

Next, we recall the Euler-Maclaurin summation formula.
\begin{proposition}[{\cite[Proposition 2.3]{BMRS}}]\label{EM}
	Let $a, z \in \mathbb{C}$. If $f$ is holomorphic and has rapid decay in $\{tz+a: t \in \mathbb{R}_0^{+}\}$. Then we have, for all $R \in \mathbb{N}$,
	\begin{align*}
		\sum_{n \geq 0} f(nz+a) = \frac{1}{z} \int_{a}^{a+z\infty} f(x) dx + \frac{f(a)}{2} - \sum_{r=1}^R \frac{B_{2r} z^{2r-1}}{(2r)!} f^{(2r-1)(a)} + O(1)z^{2R} \int_{a}^{a+z\infty} \left| f^{(2R+1)}(x)\right| dx,
	\end{align*}
    where $B_r$ is the Bernoulli number.
\end{proposition}

We also recall the Tauberian theorem of Ingham.
\begin{theorem}[{\cite[Theorem 1.1]{BJM}}]\label{Tauberian}
	Suppose that $B(q)=\sum_{n \geq 0} b_n q^n$ is a power series with non-negative real coefficients and radius of convergence at least one. If $\alpha$, $\beta$, and $\gamma$ are real numbers with $\gamma > 0$ such that
	\[
		B\left(e^{-t}\right) \sim \alpha t^{\beta} e^{\frac{\gamma}{t}} \quad \text{as } t \to 0^+, \qquad B\left(e^{-z}\right) \ll |z|^{\beta} e^{\frac{\gamma}{|z|}} \quad \text{as } z \to 0,
	\]	
	with $z=x+iy$, where $x, y \in \mathbb{R}$ with $x>0$ and in each region of the form $y \leq \Delta x$ for $\Delta>0$, then
	\[
		\sum_{n=0}^N b_n \sim \frac{\alpha \gamma^{\frac{\beta}{2}-\frac14}}{2 \sqrt{\pi} N^{\frac{\beta}{2}+\frac14}} e^{2\sqrt{\gamma N}} \qquad \text{as } N \to \infty.
	\]
	Furthermore, if $\{b_n\}$ is weakly increasing, then
	\[
		b_n \sim \frac{\alpha \gamma^{\frac{\beta}{2}+\frac14}}{2 \sqrt{\pi} n^{\frac{\beta}{2}+\frac34}} e^{2\sqrt{\gamma n}} \qquad \text{as } n \to \infty.
	\] 
\end{theorem}

%%%%%%%%%%%%%%%%%%%%%%%%%%%%%%%%%%%%%%%%%%%%%
%          Section.   %
%%%%%%%%%%%%%%%%%%%%%%%%%%%%%%%%%%%%%%%%%%%%%
\section{$t$-hooks in two partition sets from the first Rogers--Ramanujan identity}\label{sec3}

Let us start by recalling the first Rogers--Ramanujan identity:
\[
	\sum_{n\ge 0} \frac{ q^{n^2}}{(q;q)_n} =\frac{1}{(q,q^4;q^{5})_{\infty}}.
\]
%which implies that the number of partitions of $n$ into parts differing by $2$ equals the number of partitions of $n$ into parts congruent to $\pm 1$ mod $5$.

For convenience, let
%in the set $\mathcal{R}_1$ equals the number of partitions of $n$ in the set $\mathcal{R}_2$, where
\[
	\mathcal{R}_1:=\{\lambda : \lambda_{i}-\lambda_{i+1} \ge 2 \} \quad \text{and} \quad
	\mathcal{R}_2:=\{\lambda : \lambda_{i}\equiv 1,4 \pmod{5}\}.
\]
For $j=1,2$, recall from Introduction that
\begin{align*}
	r_{j,t}(n)&:=\sum\limits_{\substack{\lambda \in \mathcal{R}_j \\ |\lambda|=n}} \widetilde{r}_{j,t} (\lambda),
\end{align*}
where $\widetilde{r}_{j, t} (\lambda)$ is the number of $t$-hooks in $\lambda \in \mathcal{R}_j$.

To study $r_{j,t}(n)$, we first need to find the bivariate generating function for  $\mathcal{R}_j$:
\[
	R_{j,t}(x,q):=\sum_{n\ge 0} x^{r_{j,t}(n)} q^{n}.
\]
We then obtain the generating function of $r_{j,t}(n)$ by differentiating $R_{j,t}(x,q)$ with respect to $x$ at $1$. Namely, 
\[
	S_{j,t}(q):=\frac{\partial } {\partial x} R_{j,t}(x,q)\bigg|_{x=1}=\sum_{n\ge 0}  r_{j,t}(n) q^n. 
\]

%%%%%%%%%%%%%%%%%%%%%%%%%%%%%%%%%%%%%%%%%%%%%
\subsection{$t=1$ case} \label{sec3.1}
We first observe from the definition of $t$-hooks that $1$-hooks are corner cells in the Young diagram, so the number of $1$-hooks equals the number of different parts. Thus,
\begin{equation*}
\widetilde{r}_{1,1} (\lambda) =\ell(\lambda) \quad \text{ and } \quad 
\widetilde{r}_{2,1} (\lambda) =d (\lambda),
\end{equation*}   
where $\ell(\lambda)$ and $d(\lambda)$ denote the number of parts and the number of different parts, respectively. 
\iffalse
\begin{align*}
	\widetilde{r}_{1,1} (\lambda) &=\ell(\lambda):=\text{the number of parts of $\lambda$},\\
	\widetilde{r}_{2,1} (\lambda) &=d (\lambda):=\text{the number of different parts of $\lambda$}.
\end{align*}
\fi
The generating function for partitions in $\mathcal{R}_1$ is 
\begin{equation*}
    \sum_{n\ge 0}\frac{q^{n^2}}{(q;q)_n},
\end{equation*}
and the index $n$ in the summand counts the number of parts, from which it follows that
\[
	R_{1,1}(x,q)=\sum_{n\ge 0} x^{r_{1,1}(n)} q^{n}  =\sum_{\lambda\in \mathcal{R}_1} x^{\ell(\lambda)} q^{|\lambda|} =\sum_{n\ge 0} \frac{ x^n q^{n^2}}{(q;q)_n}.
\]
Moreover, we see   
\begin{align*}
	R_{2,1}(x,q)=\sum_{n\ge 0} x^{r_{2,1}(n)} q^{n} 
	&=\sum_{\lambda\in \mathcal{R}_2} x^{d(\lambda)} q^{|\lambda|}  \\
	&=\prod_{n\equiv 1,4 \!\!\!\pmod{5}}  (1+xq^{n}+xq^{2n}+xq^{3n}+\cdots) \\
	&= \prod_{n\equiv 1,4 \!\!\!\pmod{5}}  \left(1+\frac{xq^{n}}{1-q^n}\right)\\
	& = \prod_{n\equiv 1,4 \!\!\!\pmod{5}} \frac{1-(1-x)q^n}{1-q^n}. 
\end{align*}

We are now ready to give the generating function $S_{j,1}(q)$ for $j=1,2$. 
\begin{proposition} \label{gen_S1}
We have
\begin{align*}
	S_{1,1}(q)&=\sum_{n\ge 0} r_{1,1}(n) q^n = \sum_{n\ge 0} \frac{ n q^{n^2} }{(q;q)_n},\\
	S_{2,1}(q)&=\sum_{n\ge 0} r_{2,1}(n) q^n = \frac{1}{(q, q^4;q^5)_{\infty}} \frac{q+q^4}{1-q^{5}}.
\end{align*}
\end{proposition}
\begin{proof}
For each $S_{j,1}(q)$, $j=1,2$, we take the derivative of $R_{j,1}(x,q)$ with respect to $x$ at $x=1$. Then,
$S_{1,1}(q)$ is straightforward. For $S_{2,1}(q)$,% and $G_{3,1}(x,q)$, we use Logarithmic differentiation. 
\begin{align*}
	\frac{\partial } {\partial x} R_{2,1}(x,q)&= \frac{\partial}{\partial x} \frac{((1-x)q, (1-x)q^4;q^5)_{\infty} }{(q,q^4;q^5)_{\infty}} \\
	&= \frac{((1-x)q, (1-x)q^4;q^5)_{\infty}}{(q,q^4;q^5)_{\infty}} \sum_{n\ge 1 \atop n\equiv 1,4 \!\!\!\pmod{5}} \frac{\partial}{\partial x} \Log (1-(1-x)q^n)\\
	&= \frac{((1-x)q, (1-x)q^4;q^5)_{\infty}}{(q,q^4;q^5)_{\infty}} \sum_{n\ge 1 \atop n\equiv 1,4 \!\!\!\pmod{5}} \frac{ q^n}{1-(1-x)q^n},
\end{align*}
from which it follows that
$$
S_{2,1}(q)=\frac{1}{(q,q^4;q^5)_{\infty}} \sum_{n\ge 1 \atop n \equiv 1,4 \!\!\!\pmod{5}} q^n= \frac{1}{(q,q^4;q^5)_{\infty}} \frac{q+q^4}{1-q^5}, 
$$
as desired.
\end{proof}

%Similarly, applying Logarithmic differentiation to \eqref{R31}, we can get
%\begin{align}
%\sum_{n\ge 0} r_{3,1}(n) q^n  &=\frac{q}{1-q}  \sum_{n\ge 1} \frac{q^{n^2}}{(q;q)_{n-1}} \notag \\
%&= \frac{q}{1-q} \left( \sum_{n\ge 1} \frac{q^{n^2}}{(q;q)_n} -\sum_{n\ge 1} \frac{q^{n^2+n}}{(q;q)_n}\right) \notag \\ %= \frac{1}{(q^2, q^3;q^5)_{\infty}}.
%&= \frac{q}{1-q} \left( \sum_{n\ge 0} \frac{q^{n^2}}{(q;q)_n} -\sum_{n\ge 0} \frac{q^{n^2+n}}{(q;q)_n}\right) \notag \\ 
%&=\frac{q}{1-q} \left( \frac{1}{(q,q^4;q^5)_{\infty}}  -\frac{1}{(q^2,q^3;q^5)_{\infty}}\right). \label{dR31}
%\end{align}

%Here, the second inequality is trivial. 

%%%%%%%%%%%%%%%%%%%%%%%%%%%%%%%%%%%%%%%%%%%%%
\subsection{$t=2$ case}\label{sec3.2}
In this section, we will study the number of $2$-hooks. From the definition of $2$-hooks, it easily follows that the number of $2$-hooks equals the number of parts $\lambda_i$ with $\lambda_i-\lambda_{i+1}>1$ plus the number of parts appearing at least twice. Here, we assume that $\lambda_{\ell(\lambda)+1}=0$. Thus, %$\widetilde{r}_{1,2}(\lambda)$ equals the number of parts minus the number of parts equal to $1$, and $\widetilde{r}_{2,2}(\lambda)$ 
\begin{equation*}
\widetilde{r}_{1,2}(\lambda)=\ell_{>1}(\lambda) \quad \text{ and } \quad \widetilde{r}_{2,2}(\lambda)=d_{>1}(\lambda) + m_{>1} (\lambda),
\end{equation*}    
where $\ell_{>1}(\lambda), d_{>1}(\lambda)$ and $m_{>1}(\lambda)$ denote the number of parts greater than $1$, the number of different parts greater than $1$, and the number of parts with multiplicity greater than $1$, respectively. 
\iffalse
\begin{align*}
	\widetilde{r}_{1,2}(\lambda)&=\ell(\lambda) -\text{the number of parts equal to $1$},\\
	\widetilde{r}_{2,2} (\lambda)&=d_1 (\lambda)+ d'_2(\lambda)\\
	:&=\text{ (the number of different parts $>1$)} +\text{ (the number of parts with multiplicity $>1$)}.%,\\
	%\widetilde{r}_{3,2} (\lambda)&= \# \{ i \,|\, \lambda_i -\lambda_{i+1}\ge 2, \lambda_{\ell+1}=0 \} +\#\{ j\, |\, \lambda'_{j}-\lambda'_{j+1} \ge 2\}. 
\end{align*}
%Here $\lambda'$ denotes the conjugate of $\lambda$. 
\fi

First, for partitions in $\mathcal{R}_1$, we can divide them into two groups: partitions with or without a part of size $1$. Then
\[
	R_{1,2}(x,q)=\sum_{n\ge 0} x^{r_{1,2}(n)} q^{n} = \sum_{n\ge 1} \frac{ x^{n-1} q^{n^2}}{(q;q)_{n-1}} +\sum_{n\ge 0} \frac{ x^n q^{n^2+n}}{(q;q)_n},
\]
where the first sum generates partitions with a part of size $1$ and the second sum generates partitions with no part of size $1$. 

Also, for $\mathcal{R}_2$, applying the same analysis seen in Section~\ref{sec3.1}, we can obtain
\[
	R_{2,2}(x,q)=\sum_{n\ge 0} x^{r_{2,2}(n)} q^{n} = \left(1+q+ \frac{xq^2}{1-q} \right)  \prod_{n\equiv 1,4 \!\!\!\pmod{5} \atop n>1 } \left( 1+ xq^n + \frac{x^2 q^{2n}}{1-q^n} \right).
\]

Similarly to Proposition \ref{gen_S1}, we can derive the generating functions of $r_{1,2}(n)$ and $r_{2,2}(n)$ from the above generating functions. We omit the details.

\begin{proposition}\label{gen_S2}
We have
\begin{align*}
	S_{1,2}(q)=\sum_{n\ge 0} r_{1,2}(n) q^n 
	%& =\sum_{n\ge 1 } \frac{ n q^{n^2+n}}{(q;q)_n} + \sum_{n\ge 1} \frac{ (n-1) q^{n^2}}{(q;q)_{n-1}}\\
	& =\sum_{n\ge 1} \frac{nq^{n^2}}{(q;q)_n} - \sum_{n\ge 1} \frac{ q^{n^2}}{(q;q)_{n-1}}, \\
	S_{2,2}(q)=\sum_{n\ge 0} r_{2,2}(n) q^n  
	&=\frac{1}{(q,q^4;q^5)_{\infty}} \left( \frac{q^4+q^6}{1-q^5}+ \frac{q^2+q^{8}}{1-q^{10}} \right).
\end{align*}
\end{proposition}

%%%%%%%%%%%%%%%%%%%%%%%%%%%%%%%%%%%%%%%%%%%%%
%          Section.     %
%%%%%%%%%%%%%%%%%%%%%%%%%%%%%%%%%%%%%%%%%%%%%
\section{Asymptotics for hook numbers arising from the first Rogers--Ramanujan identity}\label{sec4}
In this section, we derive the asymptotic formulas in Theorems \ref{thm:r1_asymp} and \ref{thm:r2_asymp}.
As the proof of Theorem \ref{thm:r1_asymp} consists of several steps, we first outline the key steps and main ideas of the proof.

Let us recall the generating function of $r_{1,1}(n)$ from Proposition \ref{gen_S1}:
\[
	S_{1,1}(q)=\sum_{n\ge 0} r_{1,1}(n) q^n = \sum_{n\ge 0} \frac{ n q^{n^2} }{(q;q)_n}.
\]
We note that the summand $ \frac{ n q^{n^2} }{(q;q)_n}$ is essentially unimodal and thus the contribution around the peak will dominate the asymptotic of $S_{1,1}(q)$.
The summand is near the peak when
\[
	\frac{ (n+1)q^{(n+1)^2}}{(q;q)_{n+1}} \big/ \frac{n q^{n^2}}{(q;q)_{n}} \approx 1, 
\]
i.e., $q^n \approx \phi^{-1}$, where $\phi$ is the golden ratio. 
Thus, the main contribution to $S_{1,1}(q)$ comes from the summand when $n$ is in
\[
	\mathcal{N}_{\varepsilon}:= \left\{ n : \left| n - \frac{\log (\phi)}{\varepsilon} \right| \le \varepsilon^{-\frac{3}{5}} \right\},
\]
and we decompose
\begin{equation} \label{eqn:decom}
	S_{1,1}(q) = \mathcal{S}_{1,1}(q) + \mathcal{E}_{1,1} (q) :=
 \sum_{n \in \mathcal{N}_{\varepsilon}} \frac{nq^{n^2}}{(q;q)_n} + \sum_{n \in \mathbb{N} \setminus \mathcal{N}_{\varepsilon}}  \frac{nq^{n^2}}{(q;q)_n}. 
\end{equation}
\begin{remark}
In the definition of $\mathcal{N}_\varepsilon$, we choose the exponent of $\varepsilon$ for computational simplicity. 
We can choose $\lambda \in (-5/8,-1/2)$ and consider 
\[
	\mathcal{N}_{\varepsilon, \lambda}:= \left\{ n : \left| n - \frac{\log (\phi)}{\varepsilon} \right| \le \varepsilon^\lambda \right\}. 
\]
\end{remark}

To apply Ingham's Tauberian theorem, we need to estimate $S_{1,1}(q)$ as $z=\varepsilon(1+iy) \to 0$ in the bounded cone $|y|\ll 1$. In Section~\ref{sec:S_narrow}, we first obtain the asymptotic behavior of the summand near the peak with a narrow range of $y$. While this range of $y$ is not wide enough to employ the Tauberian theorem, this asymptotic formula is convenient to apply the saddle point method to derive
 \begin{equation} \label{S11_asym}
	\mathcal{S}_{1,1}(q) \sim \frac{\sqrt{\phi} }{5^{1/4}}\log(\phi)  z^{-1} e^{\frac{\pi^2}{15 z}} 
\end{equation}
as $z \to 0$. In Section \ref{sec:S_far}, we will show that the contribution from $\mathcal{E}_{1,1}(q)$ is negligible as expected. In Section~\ref{sec:S_asymp}, we show that the above asymptotic formula in \eqref{S11_asym} can be extended to a wider range $y \ll 1$.  Finally, we give the proofs of Theorems \ref{thm:r1_asymp} and \ref{thm:r2_asymp} in Sections \ref{sec:r1_asymp} and \ref{sec:r2_asymp}.

\subsection{Asymptotics for $\mathcal{S}_{1,1}(q)$ in a narrow range of $y$}\label{sec:S_narrow}

We first prove narrow range estimates near the peak.

\begin{proposition}\label{prop:narrow}
	Let $n \in \mathcal{N}_{\varepsilon}$ and $u \in \mathbb{C}$ be such that
	\begin{equation}\label{eq:n_to_u}
		n = \frac{\log (\phi)}{z} + \frac{u}{\sqrt{z}}.
	\end{equation}
	If $y \ll \varepsilon^{\frac13+\delta}$ for some $\delta>0$, then we have, uniformly in $u$,
	\[
		\frac{n q^{n^2}}{(q;q)_n}  =  \frac{\phi \log(\phi) }{\sqrt{2\pi z}}\exp \left( \frac{\pi^2}{15z} - \frac{\sqrt{5}\phi}{2}  u^2   \right) + O\left( \varepsilon^{3\delta_1}\right) \frac{1}{\sqrt{z}} \exp \left( \frac{\pi^2}{15z} - \frac{\sqrt{5}\phi}{2}  u^2   \right),
	\]
	where $\delta_1=\min\{\delta, \frac{1}{15}\}$. 
\end{proposition}
\begin{proof}
We see that
\[
	q^n = \exp \left( -n  z \right) = \exp \left( - \log(\phi ) - u \sqrt{z}  \right) = \phi^{-1} e^{- u \sqrt{z}}.
\]
Let $\nu := \frac{u}{\sqrt{z}}$. As $n \in \mathcal{N}_{\varepsilon}$, we observe that
\begin{equation}\label{eq:nu}
	\nu = n - \frac{\log (\phi)}{z} = \left( \frac{\log (\phi)}{\varepsilon} + \mu \right) - \frac{\log (\phi)}{\varepsilon(1+i y)} = \frac{iy}{1+iy} \frac{\log(\phi)}{\varepsilon} + \mu \ll \varepsilon^{\delta-\frac23} + \varepsilon^{-\frac35} 
\end{equation}
for some $\mu$ with $|\mu|<\epsilon^{-3/5}$.
Thus, $\nu z = o(1)$ as desired. Then we apply Lemma \ref{lem:Zagier_asym} with $w = \phi^{-1}$ and Lemma \ref{lem:eta_asym} to deduce that 
\begin{align*}
	& \Log\left(\frac{1}{(q;q)_n}\right)=\Log \left(\frac{(q^{n+1};q)_\infty}{(q;q)_\infty} \right)\\
	&= - \Li_2 (\phi^{-1}) \frac{1}{z} - \left( \nu + \frac{1}{2} \right) \log(1-\phi^{-1}) - \frac{\nu^2 z}{2} \frac{\phi^{-1}}{1-\phi^{-1}} + \psi_{\phi^{-1}} (\nu, z) +\frac{\pi^{2}}{6z} + \frac{1}{2} \Log\left(\frac{z}{2\pi} \right) - \frac{z}{24} + \mathcal{E}\\
	&= \left(\frac{\pi^2}{6} - \Li_2 (\phi^{-1})\right) \frac{1}{z} +\left( \frac{2u}{\sqrt{z}} + 1 \right) \log(\phi) 
	 + \frac{1}{2} \Log\left(\frac{z}{2\pi} \right) -\frac{\phi u^2}{2} - \frac{z}{24} + \xi \left(\frac{u}{\sqrt{z}}, z \right), 
\end{align*}
where $\xi(\nu,z):=\psi_{\phi^{-1}}(\nu,z) + \mathcal{E}$. Here, $\xi(\nu,z)$ has the same asymptotic expansion as $\psi_{\phi^{-1}}(\nu,z)$ since $\mathcal{E}\ll z^L$ for arbitrary $L \in \mathbb{N}$, i.e. for $R \in \mathbb{N}$, as $z\to 0$ with $\Re(z)>0$
\begin{equation}\label{eq:xi_n}
	\xi(\nu, z) = -\sum_{r=2}^{R-1} \left( B_r(-\nu)-\delta_{r,2} \nu^2 \right) \Li_{2-r}(\phi^{-1}) \frac{z^{r-1}}{r!} + O\left( z^{R-1} \right).
\end{equation}
We note from Lemma \ref{lem:Zagier_asym} that the coefficient of $\nu^n$ in \eqref{eq:xi_n} is $O\left(z^{2n/3} \right)$ for every $n \geq 0$. For $q^{n^2}$, we have
\[
	-n^2 z = - \left( \frac{\log (\phi)}{z} + \frac{u}{\sqrt{z}} \right)^2 z = - \frac{\log^2 (\phi)}{z}  - \frac{2\log (\phi)}{\sqrt{z}}u  - u^2.
\]
Hence, we obtain that
\begin{align*}
	\frac{n q^{n^2}}{(q;q)_n} 
	&=\left( \frac{\log (\phi)}{z} + \frac{u}{\sqrt{z}}  \right) \exp\left(  - \frac{\log^2 (\phi)}{z}  - \frac{2\log (\phi)}{\sqrt{z}}u  - u^2 \right) \\
	&\quad \times\phi \sqrt{\frac{z}{2\pi}}  \exp \left( \frac{\frac{\pi^2}{15}+\log^2(\phi)}{z} + \frac{2\log(\phi)}{\sqrt{z}}u -  \frac{\phi}{2} u^2  \right) \exp\left(-\frac{z}{24} + \xi \left(\frac{u}{\sqrt{z}}, z\right) \right)\\
	%&=  \phi \left( \frac{\log(\phi) }{z} +  \frac{u}{\sqrt{z}} \right)  \sqrt{\frac{z}{2\pi}}  \exp \left( \frac{\pi^2}{15z} - \frac{2+\phi}{2}  u^2   \right) \exp\left(-\frac{z}{24} + \xi \left( \frac{u}{\sqrt{z}}, z\right) \right)\\
	&=  \left( \frac{\phi \log(\phi) }{\sqrt{2\pi z}} + \frac{\phi}{\sqrt{2\pi}} u\right) \exp \left( \frac{\pi^2}{15z} - \frac{\sqrt{5}\phi}{2}  u^2   \right) \exp\left(-\frac{z}{24} + \xi \left( \frac{u}{\sqrt{z}}, z\right) \right),
\end{align*}
where we use
\[
	\frac{\pi^2}{6} - \Li_2 (\phi^{-1}) = \frac{\pi^2}{15} + \log^2(\phi).
\]

%Next, consider 
%\[
%	\sum_{r\geq 0} C_r(u) z^{\frac{r}{2}}:= \exp\left(-\frac{z}{24} + \xi \left( \frac{u}{\sqrt{z}}, z\right) \right).
%\]
As $y \ll \varepsilon^{\frac13+\delta}$, we have $z \asymp \varepsilon$. Since $n \in \mathcal{N}_{\varepsilon}$, we have $\frac{u}{\sqrt{z}} \ll \varepsilon^{-\frac35}$, and it follows from \eqref{eq:nu} that 
\begin{equation}\label{eq:u_bound}
	u \ll \varepsilon^{-\frac16+\delta_1}.
\end{equation}
Thus, \eqref{eq:xi_n} gives
\[
	-\frac{z}{24} + \xi \left( \frac{u}{\sqrt{z}}, z\right) =  O(z) + O\left(u\sqrt{z}\right) + \max_{n \geq 3} O\left(u^n z^{n/6}\right) =O\left(\varepsilon^{3\delta_1}\right),
\]
which holds uniformly in $u$ since the constant in the error term is independent of $u$.

Therefore, we arrive at
\begin{align*}
	\frac{n q^{n^2}}{(q;q)_n} &=  \left( \frac{\phi \log(\phi) }{\sqrt{2\pi z}} + \frac{\phi}{\sqrt{2\pi}} u\right)  \left(1+ O\left( \varepsilon^{3\delta_1}\right)\right)\exp \left( \frac{\pi^2}{15z} - \frac{\sqrt{5}\phi}{2}  u^2   \right)\\
	&=\frac{\phi \log(\phi) }{\sqrt{2\pi z}}\exp \left( \frac{\pi^2}{15z} - \frac{\sqrt{5}\phi}{2}  u^2   \right) + O\left( \varepsilon^{3\delta_1}\right) \frac{1}{\sqrt{z}} \exp \left( \frac{\pi^2}{15z} - \frac{\sqrt{5}\phi}{2}  u^2   \right). \qedhere
\end{align*}
\end{proof}

We first derive the asymptotic of $S_{1,1}(q)$ for $y$ small.

\begin{proposition}\label{prop:S11_narrow}
	If $y \ll \varepsilon^{\frac{2}{5}+\delta}$ for some $\delta>0$, then we have, as $z \to 0$, 
	\[
		\mathcal{S}_{1,1}(q) = \sum_{n \in \mathcal{N}_{\varepsilon}} \frac{nq^{n^2}}{(q;q)_n} = \frac{\sqrt{\phi} }{5^{1/4}}\log(\phi)  z^{-1} e^{\frac{\pi^2}{15 z}} + O\left(\varepsilon^{-\frac{9}{10}} e^{\frac{\pi^2}{15\varepsilon}}\right).
	\]
\end{proposition}
To prove Proposition \ref{prop:S11_narrow}, we need the following lemmas, which are analogues of  \cite[Lemmas 5.2 and 5.3]{BMRS}. Since the proofs follow from straightforward modifications, we omit the proofs here. Let $\mathcal{U}_{\varepsilon}$ be the bijective image of $\mathcal{N}_{\varepsilon}$ under the map
\[
	 n \mapsto u = - \frac{\log (\phi)}{\sqrt{z}} + n\sqrt{z}
\]
which is given in \eqref{eq:n_to_u}.

\begin{lemma}\label{lem:comp_U} 
	If $y \ll \varepsilon^{\frac{2}{5}+\delta}$ for some $\delta>0$ and $P(u)$ is a polynomial in $u$, then, as $z \to 0$, 
	\[
		 \int_{\left(-\frac{\log (\phi)}{\sqrt{z}} + \mathbb{R}\sqrt{z}\right) \setminus \mathcal{U}_{\varepsilon}} \left|P(u) e^{-\frac{\sqrt{5}\phi}{2}u^2} \right| = O\left(u^L\right),
	\]
	for all $L \in \mathbb{N}$.
\end{lemma}

\begin{lemma}\label{lem:exp}
	If $y \ll \varepsilon^{\frac{2}{5}+\delta}$ for some $\delta>0$, and $u \in \mathcal{U}_{\varepsilon}$, then we have, as $z \to 0$, 
	\[	
		\left| e^{\frac{\pi^2}{15z} - \frac{\sqrt{5}\phi}{2}  u^2}\right| \leq e^{\frac{\pi^2}{15\varepsilon}}.
	\]
\end{lemma}
%\begin{proof}
%We let $\sqrt{z}:=\varepsilon_0(1+iy+0)$. Then we have $\varepsilon = \varepsilon_0^2(1-y_0^2)$ and $y=\frac{2y_0}{1-y_0}$. As $\varepsilon >0$, we have $1-y_0^2<1$, thus, $|y_0|<1$ and $y \geq 2y_0$. As $y \ll \varepsilon^{1+\lambda+\delta}$, we have $1-y_0^2 = 1+ O(\varepsilon^{2+2\lambda+2\delta})$, and $\varepsilon_0 = \varepsilon^{\frac12}\left( 1+ O(\varepsilon^{2+2\lambda+2\delta})\right)$.
%
%Next, we bound $e^{\Re\left(- \frac{\sqrt{5}\phi}{2}  u^2\right)}$. So, we compute
%\[
%	\Re\left(- \frac{\sqrt{5}\phi}{2}  u^2\right) = \frac{\sqrt{5}\phi}{2} \left( \Im(u)^2 - \Re(u)^2\right).
%\]
%\end{proof}

We now give the proof of Proposition \ref{prop:S11_narrow}.
\begin{proof}[Proof of Proposition \ref{prop:S11_narrow}]
By Proposition \ref{prop:narrow} with $\delta_1=\frac{1}{15}$,
\begin{align*}
	\sum_{n \in \mathcal{N}_{\varepsilon}} \frac{nq^{n^2}}{(q;q)_n} =& \frac{\phi \log(\phi) }{\sqrt{2\pi z}} e^{\frac{\pi^2}{15z} } \sum_{u \in \mathcal{U}_{\varepsilon}}e^{- \frac{\sqrt{5}\phi}{2}  u^2} + O\left( \varepsilon^{\frac{1}{5}}\right) \frac{1}{\sqrt{z}} \sum_{u \in \mathcal{U}_{\varepsilon}} e^{\frac{\pi^2}{15z} - \frac{\sqrt{5}\phi}{2}  u^2}.
\end{align*}
As the summation over $\mathcal{U}_{\varepsilon}$ contains $\ll \varepsilon^{-\frac35}$ terms, it follows from Lemma \ref{lem:exp} that the error term is bounded as
\[
	\varepsilon^{\frac15} \frac{1}{\sqrt{z}} \sum_{u \in \mathcal{U}_{\varepsilon}} e^{\frac{\pi^2}{15z} - \frac{\sqrt{5}\phi}{2}  u^2} \ll \varepsilon^{\frac15-\frac35-\frac12} e^{\frac{\pi^2}{15\varepsilon}}= \varepsilon^{-\frac{9}{10}} e^{\frac{\pi^2}{15\varepsilon}}.
\]

Applying Proposition \ref{EM} with $a=-\frac{\log (\phi)}{\sqrt{z}}$, $z\to \sqrt z$, and $R=1$,
\begin{align*}
	 \sum_{u\in \mathcal{U}_{\varepsilon}} e^{-\frac{\sqrt{5}\phi}{2}u^2} =& \frac{1}{\sqrt z} \int_{-\frac{\log (\phi)}{\sqrt{z}} + \mathbb{R}\sqrt{z}} e^{-\frac{\sqrt{5} \phi}{2}u^2} du + \left(\frac12 - \frac{\sqrt{5}\phi \log(\phi)}{12}\right) e^{-\frac{\sqrt{5} \phi \log^2(\phi)}{2z}} \\
	 &+ O(1)z \int_{-\frac{\log (\phi)}{\sqrt{z}} + \mathbb{R}\sqrt{z}} \left| \left(15\phi^2 u - 5\sqrt{5} \phi^3 u^3\right) e^{-\frac{\sqrt{5}\phi}{2}u^2} \right| du.
\end{align*}
As the integrand in the first integral is holomorphic and has rapid decay, we can shift the path and have
\[
	\int_{-\frac{\log (\phi)}{\sqrt{z}} + \mathbb{R}\sqrt{z}} e^{-\frac{\sqrt{5} \phi}{2}u^2} du  =  \int_{\mathbb{R}}  e^{-\frac{\sqrt{5} \phi}{2}u^2} du =\frac{ \sqrt{2\pi}}{5^{\frac14} \phi^{\frac12}}. 
\]
For the second integral, we use Lemma \ref{lem:comp_U} to obtain that for all $L \in \mathbb{N}$,
\[
	\int_{-\frac{\log (\phi)}{\sqrt{z}} + \mathbb{R}\sqrt{z}} \left| \left(15\phi^2 u - 5\sqrt{5} \phi^3 u^3\right) e^{-\frac{\sqrt{5}\phi}{2}u^2} \right| du = \int_{\mathcal{U}_{\varepsilon}} \left| \left(15\phi^2 u - 5\sqrt{5} \phi^3 u^3\right) e^{-\frac{\sqrt{5}\phi}{2}u^2} \right| du + O\left(\varepsilon^L\right).
\]
As $u\in \mathcal{U}_{\varepsilon}$, from \eqref{eq:u_bound}, we have $|u|\ll \varepsilon^{-\frac{1}{10}}$, and the truncated segment $\mathcal{U}_{\varepsilon}$ has length $\ll \varepsilon^{-\frac{1}{10}}$. Thus, by Lemma \ref{lem:exp},
\[
	e^{\frac{\pi^2}{15z} } \int_{\mathcal{U}_{\varepsilon}}  \left| \left(15\phi^2 u - 5\sqrt{5} \phi^3 u^3\right) e^{-\frac{\sqrt{5}\phi}{2}u^2} \right| du \ll \varepsilon^{-\frac25}e^{\frac{\pi^2}{15\varepsilon}}.
\]
%Also, 
%\[
%	\int_{-\frac{\log (\phi)}{\sqrt{z}} + \mathbb{R}\sqrt{z}}u  e^{-\frac{\sqrt{5} \phi}{2}u^2} du  =  \int_{\mathbb{R}} u e^{-\frac{\sqrt{5} \phi}{2}u^2} du = 0. 
%\]

Therefore, if $y \ll \varepsilon^{\frac{2}{5}+\delta}$ for some $\delta>0$, we have that as $z \to 0$,
\begin{align*}
	\sum_{n \in \mathcal{N}_{\varepsilon}} \frac{nq^{n^2}}{(q;q)_n} &= \frac{\phi^{\frac12} }{5^{\frac14}}\log(\phi)  z^{-1} e^{\frac{\pi^2}{15 z}}  + O\left(\frac{1}{\sqrt{z}} e^{\frac{\pi^2}{15z}-\frac{\sqrt{5}\phi \log^2(\phi)}{2z}} \right)\\
	&\quad + O\left(\sqrt{z} \varepsilon^{-\frac25} e^{\frac{\pi^2}{15\varepsilon}}\right) + O\left(\varepsilon^{L+\frac12} e^{\frac{\pi^2}{15z}}\right) + O\left(\varepsilon^{-\frac{9}{10}} e^{\frac{\pi^2}{15\varepsilon}}\right)\\
	&= \frac{\phi^{\frac12} }{5^{\frac14}}\log(\phi)  z^{-1} e^{\frac{\pi^2}{15 z}} + O\left(\varepsilon^{-\frac{9}{10}} e^{\frac{\pi^2}{15\varepsilon}}\right). \qedhere
\end{align*}
\end{proof}

%%%%%%%%%%%%%%%%%%%%%%%%%%%%%%%%%%%%%%%%%%%%%
\subsection{Error estimate far from the peak}\label{sec:S_far}

In this subsection, we give an error estimate for the sum of summands far from the peak.
\begin{proposition}\label{prop:far}
	For every $L \in \mathbb{N}$, as $\Re(z) \to 0$ in the right half-plane, then
	\[
		\left| \mathcal{E}_{1,1}(q) \right| = \left|\sum_{n \in \mathbb{N} \setminus \mathcal{N}_{\varepsilon}}  \frac{nq^{n^2}}{(q;q)_n} \right| \ll \varepsilon^{L-1} e^{\frac{\pi^2}{15\varepsilon}}.
	\]
\end{proposition}
\begin{proof}
Since
\[
	\left| \frac{nq^{n^2}}{(q;q)_n} \right| \leq \frac{n|q|^{n^2}}{(|q|;|q|)_n}, 
\]
it is enough to consider the case when $z$ is real.
Thus, we suppose that $z=\varepsilon$ is real, i.e., $y=0$. Fix $n \in \mathbb{N} \setminus \mathcal{N}_{\varepsilon}$, and let $\mu \in \mathbb{R}$ be such that $n=\frac{\log(\phi)}{\varepsilon}+\mu$. Then we observe that $q^n=\frac{1}{\phi}e^{-\mu \varepsilon}$, and
\[
	\log\left((q^{n+1};q)_\infty \right) = \sum_{m \geq 1} \log\left(1-q^{n+m}\right) = -\sum_{m \geq 1}\sum_{k \geq 1}\frac{q^{k(n+m)}}{k} 
	%= -\sum_{k \geq 1} \frac{1}{k} \frac{q^{nk}}{q^{-k}-1} 
	= -\sum_{k \geq 1} \frac{1}{k \cdot \phi^k}\frac{e^{-k\mu\varepsilon}}{e^{k\varepsilon}-1}.
\]
As $e^x> 1+x$ for all $x \neq 0$ and $\frac{1}{e^x-1}>\frac{1}{x}-\frac{1}{2}$ for $x>0$, we have
\[
	\frac{e^{-k\mu\varepsilon}}{e^{k\varepsilon}-1} > (1-k\mu\varepsilon)\left(\frac{1}{k\varepsilon} -\frac12 \right) = \frac{1}{k\varepsilon} - \left( \mu +\frac12 \right) + \frac{k\mu \varepsilon}{2},
\]
which implies
\[
	\log\left((q^{n+1};q)_\infty \right)  < -\frac{\Li_2(\phi^{-1})}{\varepsilon} + \log(\phi)(2\mu+1) -\frac{\phi}{2} \mu\varepsilon.
\]
Thus, together with Lemma \ref{lem:eta_asym}, we obtain
\[
	\log \left(\frac{1}{(q;q)_n}\right)=\log\left(\frac{(q^{n+1};q)_\infty}{(q;q)_\infty} \right) <  \left(\frac{\pi^2}{6}-\Li_2(\phi^{-1})\right)\frac{1}{\varepsilon}  + \frac{1}{2} \log\left(\frac{\phi^2 \varepsilon}{2\pi} \right) - \frac{\varepsilon}{24} + 2\log(\phi) \mu -\frac{\phi}{2}\mu \varepsilon + O(\varepsilon^L). 
\]
Using
\[
	-n^2 \varepsilon = - \left( \frac{\log (\phi)}{\varepsilon} + \mu \right)^2 \varepsilon 
	= - \frac{\log^2 (\phi)}{\varepsilon}  - 2\log (\phi) \mu  - \mu^2 \varepsilon,
\]
we find 
\[
	\log\left(\frac{q^{n^2}}{(q;q)_n} \right) < \frac{\pi^2}{15\varepsilon} - \frac{1}{2} \log\left(\frac{2\pi}{\phi^2 \varepsilon} \right)  - \frac{\varepsilon}{24}  -\frac{\phi}{2} \mu\varepsilon - \mu^2 \varepsilon + O(\varepsilon^L). 
\]
If $\delta>0$, $\ell< -\frac12-\delta$, and $|\mu|>\varepsilon^\ell$, then there exists $\varepsilon_0>0$, independent of $\ell$, such that, for $\varepsilon < \varepsilon_0$,
\[
	- \frac{1}{2} \log\left(\frac{2\pi}{\phi^2 \varepsilon} \right)  - \frac{\varepsilon}{24}  -\frac{\phi}{2} \mu\varepsilon - \mu^2 \varepsilon + O(\varepsilon^L) < -\frac12 \varepsilon^{2\ell+1}.
\]
Hence, for $\varepsilon < \varepsilon_0$,
\begin{equation}\label{eq:summand_bound}
	\frac{q^{n^2}}{(q;q)_n} < e^{\frac{\pi^2}{15\varepsilon} -\frac{\varepsilon^{2\ell+1}}{2}}. 
\end{equation}

Now, we consider the sum
\[
	 \sum_{n \in \mathbb{N} \setminus \mathcal{N}_{\varepsilon}}  \frac{\varepsilon nq^{n^2}}{(q;q)_n} = \sum_{r\geq 1}\sum_{\substack{n \in \mathbb{N} \setminus \mathcal{N}_{\varepsilon}\\ \varepsilon^{-r+1}<|\mu|\leq \varepsilon^{-r}}} \frac{\varepsilon n q^{n^2}}{(q;q)_n} := \sum_{r\geq 1} A(r).
\]
For $A(1)$, the sum is over $n \in \mathbb{N}$ with $\varepsilon^{-\frac35}<|\mu|\leq \varepsilon^{-1}$. Also, $\varepsilon n=\log(\phi)+\mu \varepsilon \ll 1$. As $A(1)$ has $O(\varepsilon^{-1})$ terms, it follows from \eqref{eq:summand_bound} that
\[
	A(1) \ll \frac{1}{\varepsilon}e^{\frac{\pi^2}{15\varepsilon} -\frac{\varepsilon^{-\frac15}}{2}}.
\]
Next, we consider the summand of $A(r)$ for $r\geq 2$. As $|\mu|>\varepsilon^{-r+1}$, \eqref{eq:summand_bound} gives
\[
	\frac{q^{n^2}}{(q;q)_n} < e^{\frac{\pi^2}{15\varepsilon} -\frac{\varepsilon^{3-2r}}{2}}.
\]
As $|\mu|\leq \varepsilon^{-r}$,  $\varepsilon n=\log(\phi)+\mu \varepsilon \ll \varepsilon^{1-r}$. Since $A(r)$ contains $O(\varepsilon^{-r})$ terms, we arrive at
\[
	A(r) \ll \frac{1}{\varepsilon^{2r-1}}e^{\frac{\pi^2}{15\varepsilon} -\frac{\varepsilon^{3-2r}}{2}}.
\]

Therefore, we conclude that for all $L \in \mathbb{N}$, as $\varepsilon \to 0$,
\[
	\sum_{n \in \mathbb{N} \setminus \mathcal{N}_{\varepsilon}}  \frac{\varepsilon nq^{n^2}}{(q;q)_n} = \sum_{r\geq 1} A(r) \ll e^{\frac{\pi^2}{15\varepsilon}} \left( \frac{e^{ -\frac{\varepsilon^{-\frac15}}{2}}}{\varepsilon}+ \sum_{r\geq 2} \frac{e^{-\frac{\varepsilon^{3-2r}}{2}}}{\varepsilon^{2r-1}} \right) \ll \varepsilon^L e^{\frac{\pi^2}{15\varepsilon}}. \qedhere
\]
\end{proof}

%%%%%%%%%%%%%%%%%%%%%%%%%%%%%%%%%%%%%%%%%%%%%
\subsection{Asymptotic of $S_{1,1}(q)$}\label{sec:S_asymp}

Now we show that the asymptotic of $S_{1,1}(q)$ in Proposition \ref{prop:S11_narrow} is valid for $|y|\ll 1$. To this end, we first derive the asymptotic formula for the summand which is valid in a necessary range. 
\begin{proposition}\label{prop:wider}
	Let $n \in \mathcal{N}_{\varepsilon}$ and $v \in \mathbb{C}$ be such that
	\[
		n = \frac{\log(\phi)}{\varepsilon} + \frac{v}{\sqrt{z}}. 
	\]
	If $y \ll 1$, then we have, uniformly in $v$,
	\begin{multline*}
		\frac{n q^{n^2}}{(q;q)_n} = \frac{\log(w^{-1})}{\sqrt{2\pi z(1-w)}} \exp\left(\frac{\Lambda (y)}{z} + \Log\left(\frac{w^2}{1-w}\right)\frac{v}{\sqrt{z}} - \frac{2-w}{2(1-w)}  v^2   \right) \\
		+ O\left(\varepsilon^{\frac{1}{5}}\right) \frac{1}{\sqrt{z}} \exp\left(\frac{\Lambda (y)}{z} + \Log\left(\frac{w^2}{1-w}\right)\frac{v}{\sqrt{z}} - \frac{2-w}{2(1-w)}  v^2   \right),
	\end{multline*}
	where $w := \phi^{-1-iy}$ and
	\begin{equation}\label{eq:Lamma}
		\Lambda(y) := \frac{\pi^2}{6} - \log^2(\phi) (1+iy)^2 - \Li_2 ( \phi^{-(1+iy)} ).
	\end{equation}
\end{proposition}
\begin{proof}
We first observe that
\[
	q^n = \exp \left( -n  z \right) = \exp \left( - \log(\phi) (1+i y) - v \sqrt{z}  \right) = \phi^{-(1+iy)}e^{-v \sqrt{z}}.
\]
We have that $\nu := \frac{v}{\sqrt{z}} \ll \varepsilon^{-\frac{3}{5}}$ by the definition of $\mathcal{N}_{\varepsilon}$.
As $y \ll 1$, we see that $|z| = \varepsilon \sqrt{1 + y^2} \ll \varepsilon$.
Thus, $\nu z = o(1)$.
Since $|w|=|\phi^{-1-iy}| <1$, by Lemmas \ref{lem:Zagier_asym} and \ref{lem:eta_asym},
\begin{align*}
	&\Log \left(\frac{1}{(q;q)_n}\right)=\Log \left(\frac{(q^{n+1};q)_\infty}{(q;q)_\infty} \right) \\
	&= - \Li_2 (w) \frac{1}{z} - \left( \nu + \frac{1}{2} \right) \Log(1-w) - \frac{\nu^2 z}{2} \frac{w}{1-w} + \psi_w (\nu, z) +\frac{\pi^{2}}{6z} + \frac{1}{2} \Log\left(\frac{z}{2\pi} \right) - \frac{z}{24} + \mathcal{E}\\
	&=\left( \frac{\pi^2}{6} - \Li_2 (w) \right) \frac{1}{z} - \left( \frac{v}{\sqrt{z}} + \frac{1}{2} \right) \Log(1-w) 
	 + \frac{1}{2} \Log\left(\frac{z}{2\pi} \right) - \frac{w}{2(1-w)}v^2 - \frac{z}{24} + \xi_y \left(\frac{v}{\sqrt{z}}, z\right), 
\end{align*}
where $\xi_y(\nu, z):= \psi_w (\nu, z)+\mathcal{E}$ has the same asymptotic expansion as $\psi_w(\nu, z)$ since $\mathcal{E}\ll z^L$ for arbitrary $L \in \mathbb{N}$, i.e. for $R \in \mathbb{N}$, as $z\to 0$ with $\Re(z)>0$
\[
	\xi_y(\nu, z) = -\sum_{r=2}^{R-1} \left( B_r(-\nu)-\delta_{r,2} \nu^2 \right) \Li_{2-r}(\phi^{-1-iy}) \frac{z^{r-1}}{r!} + O\left( z^{R-1} \right),
\]
where the coefficient of $\nu^n$ in \eqref{eq:xi_n} is $O\left(z^{2n/3} \right)$ for every $n \geq 0$. For $q^{n^2}$,
\begin{align*}
	-n^2 z = - \left( \frac{\log (\phi)}{\varepsilon} + \frac{v}{\sqrt{z}} \right)^2 z = - \frac{\log^2 (\phi) (1+ iy)^2}{z} - 2  \log(\phi) (1+iy) \frac{v}{\sqrt{z}} - v^2. 	
\end{align*}
In sum, we arrive at 
\begin{align*}
	\frac{n q^{n^2}}{(q;q)_n} 
	&= \left( \frac{\log (\phi)}{\varepsilon} + \frac{v}{\sqrt{z}} \right) \exp\left( - \frac{\log^2 (\phi) (1+ iy)^2}{z} - 2  \log(\phi) (1+iy) \frac{v}{\sqrt{z}} - v^2 \right)\\
	&\quad \times \sqrt{\frac{z}{2\pi(1-w)}} \exp\left( \frac{ \frac{\pi^2}{6} - \Li_2 (w)}{z} - \Log(1-w)\frac{v}{\sqrt{z}}   - \frac{w}{2(1-w)}v^2 - \frac{z}{24} + \xi_y \left(\frac{v}{\sqrt{z}}, z\right) \right)\\
%	&=\frac{1}{\sqrt{2\pi(1-w)}}  \left( \frac{\log(\phi) (1+iy)}{\sqrt{z}} +  v \right)  \exp \left( \frac{\Lambda (y)}{z} + \Log\left(\frac{w^2}{1-w}\right)\frac{v}{\sqrt{z}} - \frac{2-w}{2(1-w)}  v^2   \right)\\
%	&\quad \times \exp\left(-\frac{z}{24} + \xi_y \left(\frac{v}{\sqrt{z}}, z\right) \right).\\
	&= \frac{1}{\sqrt{2\pi(1-\phi^{-1-iy})}}  \left( \frac{\log(\phi) (1+iy)}{\sqrt{z}} +  v \right) \\
	&\quad \times \exp\left(\frac{\Lambda (y)}{z} + \Log\left(\frac{w^2}{1-w}\right)\frac{v}{\sqrt{z}} - \frac{2-w}{2(1-w)}  v^2   \right) \exp\left(-\frac{z}{24} + \xi_y \left(\frac{v}{\sqrt{z}}, z\right) \right).
\end{align*}

As $y \ll 1$, we have $z \ll \varepsilon$. From $n \in \mathcal{N}_{\varepsilon}$, we have $v \ll \varepsilon^{-\frac{1}{10}}$. Thus, from the expansion of $\xi_y (\nu, z)$, we obtain
\[
	-\frac{z}{24} + \xi_y \left( \frac{v}{\sqrt{z}}, z\right) =  O\left(\varepsilon^{\frac{1}{5}}\right),
\]
which holds uniformly in $v$ since the constant in the error term is independent of $v$. Therefore,
\[
	\exp\left(-\frac{z}{24} + \xi_y \left( \frac{v}{\sqrt{z}}, z\right)\right) =  1+O\left(\varepsilon^{\frac{1}{5}}\right)
\] 
yields the desired asymptotic.
\end{proof}

In the following lemma, we show that the contribution of $\Lambda(y)$ is dominated by $\frac{\pi^2}{15}$, where $\Lambda(y)$ is defined in \eqref{eq:Lamma}.

\begin{lemma}\label{lem:s}
	Let $s(y):=\Re\left(\frac{\Lambda(y)}{1+iy} - \frac{\pi^2}{15}\right)$. Then we have the following.
	\begin{enumerate}
		\item $s(y) \leq 0$ for all $y \in \mathbb{R}$, and the equality holds if and only if $y=0$.
		\item As $y \to 0$,
		\[
			s(y) = - \left(\frac{\pi^2}{15} + \left( 3- \frac{\phi}{2}\right)\log^2(\phi) \right)y^2 + O\left(y^4\right).
		\]
	\end{enumerate}
\end{lemma}
\begin{proof}
(1) Recall that $\frac{\pi^2}{6} - \Li_2 (\phi^{-1}) = \frac{\pi^2}{15} + \log^2(\phi)$ and
\[
	\Lambda(y) = \Li_2 (1) - \log^2(\phi) (1+iy)^2 - \Li_2 ( \phi^{-(1+iy)} ).
\]
Since
\[
	\Li_2 (\phi^{-1-iy}) = \sum_{n \ge 1} \frac{ (1/\phi)^{(1+iy)n} }{n^2} = \sum_{n \ge 1} \frac{(1/\phi)^{i y n}}{\phi^n n^2} = \sum_{n \ge 1} \frac{\exp \left( -  \log(\phi) i y n \right)}{\phi^n n^2},
\]
we find 
\[
	s(y) = - \frac{\pi^2}{6} \frac{y^2}{1+y^2} - \frac{1}{1+y^2} \sum_{n \ge 1} \frac{ \cos \left( \log(\phi) y n \right)}{\phi^n n^2} + \frac{y}{1+y^2} \sum_{n \ge 1} \frac{ \sin \left( \log(\phi) y n \right)}{\phi^n n^2}  + \Li_2 (\phi^{-1}).
\]
As $s(-y)=s(y)$ and $s(0)=0$, it is enough to show that $f(y):= (1+y^2)s(y) \leq 0$ for $y > 0$.

First, using the fact that $|\sin(x)|\leq 1$, $|\cos(x)|\leq 1$, $\sin(x) \leq x$, and $\cos(x) \geq 1-\frac{x^2}{2}$  for all $x \in \mathbb{R}$, we observe
\begin{align*}
	f(y)%&= \frac{\pi^2}{6} -  \sum_{n \ge 1} \frac{ \cos \left( \log(\phi) y n\right)}{\phi^n n^2} +y \sum_{n \ge 1} \frac{ \sin \left( \log(\phi) y n \right)}{\phi^n n^2}  -(1+y^2) \left( \frac{\pi^2}{15}+ \log^2 (\phi)  \right)\\
	%&= \frac{\pi^2}{6} -  \sum_{n \ge 1} \frac{ \cos \left( \log(\phi) y n\right)}{\phi^n n^2} +y \sum_{n \ge 1} \frac{ \sin \left( \log(\phi) y n \right)}{\phi^n n^2}  -(1+y^2) \left( \frac{\pi^2}{6}- \Li_2 (\phi^{-1})  \right)\\
	%&= -  \sum_{n \ge 1} \frac{ \cos \left( \log(\phi) y n\right)}{\phi^n n^2} + y \sum_{n \ge 1} \frac{ \sin \left( \log(\phi) y n \right)}{\phi^n n^2} -\frac{\pi^2}{6} y^2 + \Li_2 (\phi^{-1})(1+y^2) \\
	&= -\frac{\cos(\log(\phi)y)}{\phi} -  \sum_{n \ge 2} \frac{ \cos \left( \log(\phi) y n\right)}{\phi^n n^2} +y \frac{\sin(\log(\phi)y)}{\phi}+ y \sum_{n \ge 2} \frac{ \sin \left( \log(\phi) y n \right)}{\phi^n n^2}\\
	&\quad -\frac{\pi^2}{6} y^2 + \Li_2 (\phi^{-1})(1+y^2) \\
	&\leq  -\frac1{\phi}\left(1-\frac{\log^2(\phi)}{2} y^2\right) +  \sum_{n \ge 2} \frac{1}{\phi^n n^2} +\frac{\log(\phi)}{\phi}y^2 + y \sum_{n \ge 2} \frac{1}{\phi^n n^2}  -\frac{\pi^2}{6} y^2 + \Li_2 (\phi^{-1})(1+y^2) \\
	%&= -\frac{\cos(\log(\phi)y)}{\phi} + \Li_2 (\phi^{-1}) +y \frac{\sin(\log(\phi)y)}{\phi}+ (1+y) \left( \Li_2(\phi^{-1}) - \frac1{\phi}  \right) -y^2 \left( \frac{\pi^2}{6}- \Li_2 (\phi^{-1})  \right)\\
	%&=  -\frac1{\phi}\left(1-\frac{\log^2(\phi)}{2} y^2\right)  +\frac{\log(\phi)}{\phi}y^2+ (1+y) \left( \Li_2(\phi^{-1}) - \frac1{\phi}  \right) -\frac{\pi^2}{6} y^2 + \Li_2 (\phi^{-1})(1+y^2) \\
	&= (2+y) \left( \Li_2(\phi^{-1}) - \frac1{\phi}  \right) -y^2 \left( \frac{\pi^2}{6}- \Li_2 (\phi^{-1}) - \frac{\log^2(\phi)}{2\phi} -\frac{\log(\phi)}{\phi} \right),
\end{align*}
which is negative for  $y \geq 1$.
We also have
\begin{align*}
	f'(y) =& \log(\phi) \sum_{n \ge 1} \frac{ \sin \left( \log(\phi) y n\right)}{\phi^n n} +  \sum_{n \ge 1} \frac{ \sin \left( \log(\phi) y n \right)}{\phi^n n^2} + y \log(\phi) \sum_{n \ge 1} \frac{ \cos \left( \log(\phi) y n \right)}{\phi^n n}\\
	& -2 y \left( \frac{\pi^2}{6}- \Li_2 (\phi^{-1})  \right),
\end{align*}
and 
\begin{align*}
	f''(y) &= \log^2(\phi) \sum_{n \ge 1} \frac{ \cos \left( \log(\phi) y n\right)}{\phi^n} +2\log(\phi)\sum_{n \ge 1} \frac{ \cos \left( \log(\phi) y n \right)}{\phi^n n} - y \log^2(\phi) \sum_{n \ge 1} \frac{ \sin \left( \log(\phi) y n \right)}{\phi^n} \\
	&\quad -2 \left( \frac{\pi^2}{6}- \Li_2 (\phi^{-1})  \right)\\
	& \leq  \log^2(\phi) \sum_{n \ge 1} \frac{1}{\phi^n} + 2\log(\phi) \sum_{n \ge 1} \frac{1}{\phi^n n} + y \log^2(\phi)\sum_{n \ge 1}\frac{1}{\phi^n} -2 \left( \frac{\pi^2}{15}+ \log^2 (\phi)  \right) \\
	%&= \phi \log^2(\phi) +4\log^2(\phi)+y \phi\log^2(\phi)- 2\left( \frac{\pi^2}{15}+ \log^2 (\phi)  \right)\\
	& = y \phi\log^2(\phi) + (\phi+2) \log^2(\phi)  -\frac{2\pi^2}{15}.
\end{align*}
Therefore, for $0 < y \leq 1$,
\[
	f''(y)  \leq y \phi\log^2(\phi) + (\phi+2) \log^2(\phi)  -\frac{2\pi^2}{15} < 0, 
\]
from which we deduce
\[
	f(y) < f(0) + f'(0) y =0.
\]

(2) This can be obtained from the Taylor series expansion at $y=0$.
\end{proof}

Now, we are ready to prove the asymptotic of $S_{1,1}(q)$ as $z=\varepsilon(1+iy) \to 0$ in the bounded cone $|y|\ll 1$.
\begin{proposition}\label{prop:S11_asymp}
	If $y \ll 1$, then we have that as $z \to 0$,
	\[
		S_{1,1}(q) = \frac{\phi^{\frac12} }{5^{\frac14}}\log(\phi)  z^{-1} e^{\frac{\pi^2}{15 z}} +  O\left(\varepsilon^{-\frac{9}{10}} e^{\frac{\pi^2}{15\varepsilon}}\right).
	\]
\end{proposition}
\begin{proof}
By Proposition \ref{prop:far}, we have, for all $L \in \mathbb{N}$, as $z\to 0$,
\[
	\left|\sum_{n \in \mathbb{N} \setminus \mathcal{N}_{\varepsilon}}  \frac{nq^{n^2}}{(q;q)_n} \right| \ll \varepsilon^{L-1} e^{\frac{\pi^2}{15\varepsilon}}.
\]
Thus, the desired asymptotic follows from the following claim:
\[
	z e^{-\frac{\pi^2}{15\varepsilon}} \left(\sum_{n \in \mathcal{N}_{\varepsilon}} \frac{nq^{n^2}}{(q;q)_n} - \frac{\phi^{\frac12} }{5^{\frac14}}\log(\phi)  z^{-1} e^{\frac{\pi^2}{15 z}} \right) \ll \varepsilon^{\frac{1}{10}} .
\]
We verify this claim by splitting the range of $y$ into two pieces, $y \ll \varepsilon^{\frac25+\delta}$ and  $\varepsilon^{\frac12-\delta}\ll y \ll 1$. By choosing $\delta >0$ sufficiently small, these two pieces can cover all necessary range $y \ll 1$. We first note that for $y \ll \varepsilon^{\frac25+\delta}$ for some $\delta>0$, then, by Proposition \ref{prop:S11_narrow}, we have
\[
	\sum_{n \in \mathcal{N}_{\varepsilon}} \frac{nq^{n^2}}{(q;q)_n} - \frac{\phi^{\frac12} }{5^{\frac14}}\log(\phi)  z^{-1} e^{\frac{\pi^2}{15 z}} \ll \varepsilon^{-\frac{9}{10}} e^{\frac{\pi^2}{15\varepsilon}}.
\]

For the remaining range $\varepsilon^{\frac12-\delta}\ll y \ll 1$ for some $\delta>0$, we will show that, for all $L \in \mathbb{N}$,
\begin{equation}\label{eq:first_part}
	z e^{-\frac{\pi^2}{15\varepsilon}}\sum_{n \in \mathcal{N}_{\varepsilon}} \frac{nq^{n^2}}{(q;q)_n} = O\left(\varepsilon^L\right).
\end{equation}
If $y \ll 1$, by Proposition \ref{prop:wider},
\begin{multline*}
	z e^{-\frac{\pi^2}{15\varepsilon}}\sum_{n \in \mathcal{N}_{\varepsilon}} \frac{nq^{n^2}}{(q;q)_n} = \log(w^{-1})\sqrt{\frac{z}{2\pi(1-w)}} e^{\frac{1}{\varepsilon}\left(\frac{\Lambda(y)}{1+iy}-\frac{\pi^2}{15}\right)} \sum_{n \in \mathcal{N}_{\varepsilon}} e^{-\frac{2-w}{2(1-w)}v^2 +\frac{v}{\sqrt{z}} \Log\left(\frac{w^2}{1-w}\right)}\\
	+ O\left(\varepsilon^{\frac15}\right) \sqrt{z} e^{\frac{1}{\varepsilon}\left(\frac{\Lambda(y)}{1+iy}-\frac{\pi^2}{15}\right)} \sum_{n \in \mathcal{N}_{\varepsilon}}  e^{-\frac{2-w}{2(1-w)}v^2 +\frac{v}{\sqrt{z}} \Log\left(\frac{w^2}{1-w}\right)},
\end{multline*}
where $w=\phi^{-1-iy}$. 
Thus, to prove \eqref{eq:first_part}, it is sufficient to show that the exponent
\begin{equation}\label{eq:exp}
	\frac{1}{\varepsilon}\left(\frac{\Lambda(y)}{1+iy}-\frac{\pi^2}{15}\right) -\frac{2-w}{2(1-w)}v^2 +\frac{v}{\sqrt{z}} \Log\left(\frac{w^2}{1-w}\right)
\end{equation}
has negative real part of size $\gg \varepsilon^{-\delta_0}$ for some $\delta_0>0$ since the other terms are bounded as $z\to 0$. 

If $\varepsilon^{\frac12-\delta} \ll y \ll 1$ for some $\delta>0$, then %$ \delta < \frac12$.
the real part of $\frac{1}{\varepsilon}\left(\frac{\Lambda(y)}{1+iy}-\frac{\pi^2}{15}\right)$ is $\frac{s(y)}{\varepsilon}$, which is negative by Lemma \ref{lem:s} (1). We also obtain, as $\varepsilon \to 0$, 
\[
	\frac{s(y)}{\varepsilon} \gg \frac{y^2}{\varepsilon} \gg \varepsilon^{-2\delta}
\]
by Lemma \ref{lem:s} (2). 
From the Taylor series expansion, as $y \to 0$, 
\[
	\Re\left( \Log\left(\frac{w^2}{1-w}\right) \right) = \log \left|\frac{w^2}{1-w} \right| = -\frac{2+\sqrt{5}}{2} \log^2(\phi) y^2 + O\left(y^4\right) \ll y^2.
\]
As $n \in \mathcal{N}_{\varepsilon}$, we have $\left|\frac{v}{\sqrt{z}}\right| \leq \varepsilon^{-\frac35}$, which implies
\[
	\Re\left( \frac{v}{\sqrt{z}} \Log\left(\frac{w^2}{1-w}\right) \right) \ll \varepsilon^{-\frac35} y^2.
\]	
Since $\frac{y^2}{\varepsilon} \gg \varepsilon^{-\frac35} y^2$, $\frac{s(y)}{\varepsilon}$ dominates $\Re\left( \frac{v}{\sqrt{z}} \Log\left(\frac{w^2}{1-w}\right) \right)$. For the term $\frac{2-w}{2(1-w)}v^2$, we consider the following two cases. 
\begin{enumerate}
	\item Suppose $\varepsilon^{\frac12-\delta} \ll y \ll \varepsilon^{\frac14}$. As $v$ is a real multiple of $\sqrt{z}$ and $y\ll \varepsilon^{\frac14}$, we have that for $y \ll \varepsilon^{\frac14}$,
	\[
		\Re \left(\frac{2-w}{1-w} (1+iy)\right) = \frac{2+\phi^{-2}-3\phi^{-1}\cos\left(y\log(\phi)\right)+y\phi^{-1}\sin\left(y\log(\phi)\right)}{|1-w|^2} >0.
	\]
	Thus, $-\frac{2-w}{2(1-w)}v^2$ has negative real part as $z\to 0$.
	
	\item If $\varepsilon^{\frac13} \ll y \ll 1$, then $|z|\ll \varepsilon$. As $n \in \mathcal{N}_{\varepsilon}$, we have $|v|\ll \varepsilon^{-\frac1{10}}$. Thus, with the fact $\left|\frac{2-w}{2(1-w)}\right| \leq \frac{2-\phi^{-1}}{2(1-\phi^{-1})}$, we find $-\frac{2-w}{2(1-w)}v^2 \ll  \varepsilon^{-\frac15}$. %Since $\delta> \frac14> \frac1{10}$, $\frac{s(y)}{y}$ dominates this term.
\end{enumerate}
In either case, the exponent $\frac{s(y)}{\varepsilon}$ has size $\gg \varepsilon^{-2\delta}$ and dominates the other exponents in \eqref{eq:exp} that have positive real part. Hence, we prove the claim \eqref{eq:first_part} by taking $\delta_0=2\delta$.

Next, in the expression $\frac{\phi^{\frac12} }{5^{\frac14}}\log(\phi) e^{\frac{\pi^2}{15 z}-\frac{\pi^2}{15\varepsilon}}$, the exponential part has negative real part of size
\[
	\Re\left(\frac{\pi^2}{15 z}-\frac{\pi^2}{15\varepsilon}\right)% = \frac{\pi^2}{15\varepsilon} \left( \Re\left(\frac{1}{1+iy}\right) -1\right) 
	= - \frac{\pi^2}{15\varepsilon} \frac{y^2}{1+y^2} \gg \varepsilon^{-2\delta}
\]
for $\varepsilon^{\frac12-\delta} \ll y \ll 1$. Thus, it follows that, for $\varepsilon^{\frac12-\delta} \ll y \ll 1$, we have, for all $L \in \mathbb{N}$,
\[
	\frac{\phi^{\frac12} }{5^{\frac14}}\log(\phi) e^{\frac{\pi^2}{15 z}-\frac{\pi^2}{15\varepsilon}} \ll \varepsilon^L.
\]
\end{proof}

%%%%%%%%%%%%%%%%%%%%%%%%%%%%%%%%%%%%%%%%%%%%% 
\subsection{Proof of Theorem \ref{thm:r1_asymp}}\label{sec:r1_asymp}
In this subsection, we prove Theorem \ref{thm:r1_asymp}.
 
\begin{proof}[Proof of Theorem \ref{thm:r1_asymp}]
Let $q=e^{-z}$. By Proposition \ref{prop:S11_asymp}, we have, as $z=\varepsilon(1+iy) \to 0$ for $\varepsilon>0$ and $y \ll 1$,  
\[
	S_{1,1}(q) =\sum_{n\ge 0} r_{1,1}(n) q^n \sim \frac{\sqrt{\phi} }{5^{1/4}}\log(\phi)  z^{-1} e^{\frac{\pi^2}{15 z}}.
\]

Next, we recall Proposition \ref{gen_S2}:
\[
	S_{1,2}(q)=\sum_{n\ge 0} r_{1,2}(n) q^n  =\sum_{n\ge 1} \frac{nq^{n^2}}{(q;q)_n} - \sum_{n\ge 1} \frac{ q^{n^2}}{(q;q)_{n-1}} 
	= S_{1,1}(q) - \frac{1}{(q,q^4;q^5)_\infty} + \frac{1}{(q^2,q^3;q^5)_\infty}.
\]
By \cite[Corollary 1.4]{K}, as $z \to 0$ in the right half-plane,
\[
	\frac{1}{(q,q^4;q^5)_\infty} \sim \frac{\phi^{\frac12}}{5^{\frac14}}e^{\frac{\pi^2}{15z}} \qquad \text{and} \qquad
	\frac{1}{(q^2,q^3;q^5)_\infty} \sim \frac{1}{5^{\frac14}\phi^{\frac12}}e^{\frac{\pi^2}{15z}}.
\]
Thus, $S_{1,2}(q)$ has the same asymptotic as $S_{1,1}(q)$.

From Propositions \ref{gen_S1} and \ref{gen_S2}, we have
\begin{align*}
	(1-q)S_{1,1}(q) &= \sum_{n\ge 1} \frac{nq^{n^2}}{(q^2;q)_{n-1}} \qquad \text{and}\qquad (1-q)S_{1,2}(q) = \sum_{n\ge 1 } \frac{ n q^{n^2+n}}{(q^2;q)_{n-1}} + \sum_{n\ge 2} \frac{ (n-1) q^{n^2}}{(q^2;q)_{n-2}},
\end{align*}
from which we can see that $r_{1,1}(n)$ and $r_{1,2}(n)$ are weakly increasing. Therefore, we employ Theorem \ref{Tauberian} to obtain
\[
	r_{1,1}(n) \sim r_{1,2}(n) \sim \frac{3^{\frac14} \phi^{\frac12}  \log(\phi) }{2 \pi} n^{-\frac14}e^{2\pi \sqrt{\frac{n}{15}}}
\]
as $n \to \infty$.
\end{proof}

%%%%%%%%%%%%
\subsection{Proof of Theorem~\ref{thm:r2_asymp}} \label{sec:r2_asymp}
In this subsection, we prove Theorem \ref{thm:r2_asymp}.

\begin{proof}[Proof of Theorem~\ref{thm:r2_asymp}]
We recall the generating functions of $r_{2,1}(n)$ and $r_{2,2}(n)$ from Propositions \ref{gen_S1} and \ref{gen_S2}:
\begin{align*}
	S_{2,1}(q)&=\sum_{n\ge 0} r_{2,1}(n) q^n = \frac{1}{(q, q^4;q^5)_{\infty}} \frac{q+q^4}{1-q^{5}}, \\
	S_{2,2}(q)&=\sum_{n\ge 0} r_{2,2}(n) q^n  =\frac{1}{(q,q^4;q^5)_{\infty}} \left( \frac{q^4+q^6}{1-q^5}+ \frac{q^2+q^{8}}{1-q^{10}} \right).
\end{align*}
Let $q=e^{-z}$. By \cite[Corollary 1.4]{K}, as $z \to 0$ in the right half-plane,
\begin{align*}
	\frac{1}{(q,q^4;q^5)_\infty} &= \frac{\phi^{\frac12}}{5^{\frac14}}e^{\frac{\pi^2}{15z}-\frac{z}{60}} \frac{1}{\left(e^{-\frac{4\pi^2}{5z}+\frac{2\pi i}{5}}, e^{-\frac{4\pi^2}{5z}-\frac{2\pi i}{5}}; e^{-\frac{4\pi^2}{5z}}\right)_\infty}\\
	&= \frac{\phi^{\frac12}}{5^{\frac14}}e^{\frac{\pi^2}{15z}-\frac{z}{60}} \left( 1+ O\left( e^{-\frac{4\pi^2}{5z}}\right) \right).
\end{align*}
We also obtain that as $z \to 0$,
\[
	\frac{q + q^4}{1-q^5} = \frac{2}{5z} + O(z) \qquad\text{and}\qquad  \frac{q^4+q^6}{1-q^5} + \frac{q^2+q^8}{1-q^{10}} = \frac{3}{5z} -1 + O(z).
\]
%where we use the generating function of the Bernoulli polynomials $B_n(x)$
%\begin{equation}\label{eq:B_poly}
%	\frac{t e^{xt}}{e^t-1}=\sum_{n \geq 0} \frac{B_n(x)}{n!} t^n.
%\end{equation}
Thus, as $z \to 0$ in the right half-plane,
\[
	S_{2,1}(q) = \frac{2\phi^{\frac12}}{5^{\frac54}z}e^{\frac{\pi^2}{15z}} \left( 1+ O(z)\right) \qquad \text{and}\qquad S_{2,2}(q) = \frac{3\phi^{\frac12}}{5^{\frac54}z}e^{\frac{\pi^2}{15z}} \left( 1+ O(z)\right).
\]

Since we have 
\begin{align*}
	(1-q)S_{2,1}(q) &= \frac{1}{(q^6, q^4;q^5)_{\infty}} \frac{q+q^4}{1-q^{5}},\\
	(1-q)S_{2,2}(q) &= \frac{1}{(q^6,q^4;q^5)_{\infty}} \left( \frac{q^4+q^6}{1-q^5}+ \frac{q^2+q^{8}}{1-q^{10}} \right),
\end{align*}
the $q$-series on the right-hand sides above has non-negative coefficients, and so 
$r_{2,1}(n)$ and $r_{2,2}(n)$ are weakly increasing. Hence, by Theorem \ref{Tauberian}, the desired asymptotics are derived.
\end{proof}

%%%%%%%%%%%%%%%%%%%%%%%%%%%%%%%%%%%%%%%%%%%%%
%          Section.   %
%%%%%%%%%%%%%%%%%%%%%%%%%%%%%%%%%%%%%%%%%%%%%
\section{$t$-hooks in two partition sets from the first little G\"{o}llnitz Identity}\label{sec5}

Let us now recall the first little G\"ollnitz identity:
\begin{align*}
	\sum_{n\ge 0} \frac{ q^{n^2+n} (-q^{-1};q^2)_n}{(q^2;q^2)_n} =\frac{1}{(q,q^5,q^6;q^{8})_{\infty}}.
\end{align*}   

For convenience, let
\begin{align*}
	\mathcal{G}_1&:=\{\lambda : \lambda_{i}-\lambda_{i+1} \ge 2, \;\; \lambda_{i}-\lambda_{i+1} >2 \text{ if $\lambda_i\equiv 1 \pmod{2}$}\}
\end{align*}
and
\begin{align*}
	\mathcal{G}_2&:=\{\lambda : \lambda_{i}\equiv 1, 5, 6 \pmod{8}\}.
\end{align*} 
For $j=1,2$, recall from Introduction that
\begin{align*}
	g_{j,t}(n)&=\sum\limits_{\substack{\lambda\in\mathcal{G}_j \\ |\lambda|=n}}\widetilde{g}_{j,t} (\lambda),
\end{align*}
where $\widetilde{g}_{j, t} (\lambda)$ is the number of $t$-hooks in $\lambda \in \mathcal{G}_j$. 

In this section, we study the following generating functions for $t=1,2$:
\begin{equation*}
	G_{j,t}(x,q):=\sum_{n\ge 0} x^{g_{j,t}(n)} q^{n}
\end{equation*}
and
\begin{equation*}
	H_{j,t}(q):=\frac{\partial}{\partial x} G_{j,t}(x,q) \bigg|_{x=1}=\sum_{n\ge 0} g_{j,t}(n) q^n. 
\end{equation*}

%%%%%%%%%%%%%%%%%%%%%%%%%%%%%%%%%%%%%%%%%%%%%
\subsection{$t=1$ case}\label{sec5.1}
Due to the part difference conditions on $\mathcal{G}_j$ for $j=1,2$, we can easily see that
\begin{align*}
	\widetilde{g}_{1,1} (\lambda)&=\ell(\lambda) %: =\text{ the number of parts of $\lambda$},\\
	%\widetilde{g}_{2,1}  (\lambda)&=\ell(\lambda):=\text{ number of parts of $\lambda$},\\
    \quad \text{ and } \quad 
	\widetilde{g}_{2,1}  (\lambda)=d (\lambda).
    %:=\text{ the number of different parts of $\lambda$}.
\end{align*}

In \cite[Eq. (2.20)]{andrews_siam}, Andrews gave the generating function $G_{1,1}(x,q)$:
\[
	G_{1,1}(x,q)= \sum_{n\ge 0} x^{g_{1,1}(n)} q^{n} =\sum_{\lambda\in \mathcal{G}_1}  x^{\ell(\lambda)} q^{|\lambda|} %& =\sum_{n\ge 0} \frac{ z^n q^{n^2+n} (-q;q^2)_n (1+zq^{2n+1})}{(q^2;q^2)_n} \notag \\
    =\sum_{n\ge 0} \frac{x^n q^{n^2+n}(-q^{-1};q^2)_n}{(q^2;q^2)_n}. 
\]
Also, by applying the same method used in Section~\ref{sec3}, we get
\[
	G_{2,1}(x,q)= \sum_{n\ge 0} x^{g_{2,1}(n)} q^{n}   =\sum_{\lambda\in \mathcal{G}_2} x^{d(\lambda)} q^{|\lambda|} = \prod_{n\equiv 1,5,6 \!\!\!\pmod{8}} \frac{1-(1-x)q^n}{1-q^n}.
\]

Taking the derivative of $G_{j,1}(x,q)$ with respect to $x$ at $x=1$, we get the generating functions for $g_{j,1}(n)$ for $j=1,2$. We omit the details. 
\begin{proposition} \label{gen_H1}
    We have
\begin{align*}
	H_{1,1}(q)=\sum_{n\ge 0} g_{1,1}(n)q^n &= \sum_{n\ge 0} \frac{ n q^{n^2+n} (-q^{-1};q^2)_n}{(q^2;q^2)_n}, \\
	H_{2,1}(q)=\sum_{n\ge 0} g_{2,1}(n)q^n &= \frac{1}{(q, q^5, q^6;q^8)_{\infty}} \frac{q+q^5+q^6}{1-q^{8}}. 
\end{align*}
\end{proposition}

%%%%%%%%%%%%%%%%%%%%%%%%%%%%%%%%%%%%%%%%%%%%%
\subsection{$t=2$ case}\label{sec5.2}

Note that $\widetilde{g}_{1,2}(\lambda)=\ell_{>1}(\lambda)$. It follows from Andrews' results in \cite[Eqs. (2.19), (2.20)]{andrews_siam}
\begin{align*}
	G_{1,2}(x,q)&=\sum_{n\ge 0} x^{g_{1,2} (n) } q^{n}= \sum_{n\ge 0} \frac{ x^n q^{n^2+n} (-q;q^2)_{n+1}}{(q^2;q^2)_n}.
\end{align*}
Moreover, as seen in Section~\ref{sec3.2}, 
\begin{equation*}
\widetilde{g}_{2,2}(\lambda)=d_{>1}(\lambda)+m_{>1}(\lambda). 
\end{equation*}
Thus, we have
\begin{multline*}
	G_{2,2}(x,q)=\left(1+q+\frac{xq^2}{1-q}\right) \prod_{n\ge 1} \left(1+xq^{8n+1}+ \frac{x^2 q^{2(8n+1)}}{1-q^{8n+1}} \right)
	 \prod_{n\ge 0} \left( 1+ xq^{8n+5} +\frac{x^2 q^{2(8n+5)}}{1-q^{8n+5}} \right.\\ \left. + xq^{8n+6} +\frac{x^2 q^{2(8n+6)}}{1-q^{8n+6}}
    + x q^{16n+11} \frac{1-(1-x)q^{8n+5}}{1-q^{8n+5}} \frac{1-(1-x)q^{8n+6}}{1-q^{8n+6}}\right). 
\end{multline*}
From the above generating functions, we then obtain the following.
\begin{proposition}\label{gen_H2}
We have
\begin{align*}
	H_{1,2}(q)&=\sum_{n\ge 0} g_{1,2}(n) q^{n} = \sum_{n\ge 1} \frac{n q^{n^2+n}(-q;q^2)_{n+1}}{ (q^2;q^2)_n},\\
	H_{2,2}(q)&=\sum_{n\ge 0} g_{2,2}(n) q^{n} = \frac{1}{(q,q^5,q^6;q^8)_\infty}\left( \frac{q^5+q^6+q^9}{1-q^8} + \frac{q^2+q^{10}-q^{11}+q^{12}}{1-q^{16}} \right).
\end{align*}
\end{proposition}

%%%%%%%%%%%%%%%%%%%%%%%%%%%%%%%%%%%%%%%%%%%%%
%          Section.   %
%%%%%%%%%%%%%%%%%%%%%%%%%%%%%%%%%%%%%%%%%%%%%
\section{Asymptotics for hook numbers arising from Little G\"ollnitz identity}\label{sec6}
In this section, we derive the asymptotic formulas in Theorems \ref{thm:g1_asymp} and \ref{thm:g2_asymp}.
%\begin{proof}[Proof of Theorem \ref{thm:g_ineq}]
%From the fact $\log(1+\sqrt{2}) > \frac{7}{8} > \frac{3}{4}$,
%we conclude that $g_{1,1}(n) > g_{2,1}(n)$ and $g_{1,2}(n) > g_{2,2}(n)$ for sufficiently large $n$ by Theorems \ref{thm:g1_asymp} and \ref{thm:g2_asymp}.
%\end{proof}
In the following subsections, we give the proof for the asymptotic of $g_{1,1}(n)$. Other asymptotics can be derived similarly.

As in the previous section, we first investigate when the magnitude of the summand in the generating function is large. Recall Proposition \ref{gen_H1}:
\[
	H_{1,1}(q)=\sum_{n\ge 0} g_{1,1}(n)q^n = \sum_{n\ge 0} \frac{ n q^{n^2+n} (-q^{-1};q^2)_n}{(q^2;q^2)_n}.
\]
The peak occurs at $q^{2n} \approx \sqrt{2}-1$ as
\[
	\frac{ (n+1) q^{(n+1)^2+(n+1)} (-q^{-1};q^2)_{n+1}}{(q^2;q^2)_{n+1}} \big/ \frac{ n q^{n^2+n} (-q^{-1};q^2)_n}{(q^2;q^2)_n} \approx 1.
\]
Thus, we consider
\[
	\mathcal{N}_\varepsilon := \left\{ n : \left|n-\frac{\log(\sqrt{2}+1)}{2\varepsilon} \right| \leq \varepsilon^{-\frac35} \right\}.
\]
Similarly to \eqref{eqn:decom}, we decompose the generating function as  
\[
	H_{1,1}(q) = \mathcal{H}_{1,1}(q) + \mathcal{E}_{1,1} (q) :=
 \sum_{n \in \mathcal{N}_{\varepsilon}} \frac{ n q^{n^2+n} (-q^{-1};q^2)_n}{(q^2;q^2)_n}+ \sum_{n \in \mathbb{N} \setminus \mathcal{N}_{\varepsilon}} \frac{ n q^{n^2+n} (-q^{-1};q^2)_n}{(q^2;q^2)_n}. 
\]
We first derive the asymptotic formula for $\mathcal{H}_{1,1}(q)$ in a narrow range of $y$ and later extend it to the necessary range. As the proof proceeds similarly to that of Theorem~\ref{thm:r1_asymp}, while we keep the general outline brief, we give  more computational details when technical differences occur.

%%%%%%%%%%%%%%%%%%%%%%%%%%%%%%%%%%%%%%%%%%%%%
\subsection{First Asymptotic estimate for $\mathcal{H}_{1,1}(q)$ and $\mathcal{E}_{1,1}(q)$}\label{sec:H_near}

We start with proving a narrow range estimate for $\mathcal{H}_{1,1}(q)$.
\begin{proposition}\label{prop:H_narrow}
	Let $n \in \mathcal{N}_{\varepsilon}$ and $u \in \mathbb{C}$ be such that
	\[
		n = \frac{\log(\sqrt{2}+1)}{2z} + \frac{u}{\sqrt{z}}.
	\]
	If $y \ll \varepsilon^{\frac13+\delta}$ for some $\delta>0$, then we have, uniformly in $u$,
	\[
		\frac{ n q^{n^2+n} (-q^{-1};q^2)_n}{(q^2;q^2)_n}  =  \frac{\log(\sqrt{2}+1) }{2^{\frac34}\sqrt{\pi z}}\exp \left( \frac{\pi^2}{16z} - 2  u^2   \right) + O\left( \varepsilon^{3\delta_1}\right) \frac{1}{\sqrt{z}} \exp \left( \frac{\pi^2}{16z} - 2u^2   \right),
	\]
	where $\delta_1=\min\{\delta, \frac{1}{15}\}$. 
\end{proposition}
\begin{proof}
We see that
\[
	q^{2n} = \exp \left( -2n  z \right) = \exp \left( - \log(\sqrt{2}+1) - 2u \sqrt{z}  \right) = (\sqrt{2}-1) e^{- 2u \sqrt{z}}.
\]
Let $\nu := \frac{u}{\sqrt{z}}$. As $n \in \mathcal{N}_{\varepsilon}$, we observe that
\[
	\nu = n - \frac{\log(\sqrt{2}+1)}{2z} = \left( \frac{\log(\sqrt{2}+1)}{\varepsilon} + \mu \right) - \frac{\log(\sqrt{2}+1)}{\varepsilon(1+i y)} = \frac{iy}{1+iy} \frac{\log(\sqrt{2}+1)}{\varepsilon} + \mu \ll \varepsilon^{\delta-\frac23} + \varepsilon^{-\frac35}
\]
for some $\mu$ with $|\mu|\le \epsilon^{-3/5}$.
Therefore, $\nu z = o(1)$ as desired.  We set $w = \sqrt{2}-1$. Then, $|w| <1$, and thus by Lemmas \ref{lem:Zagier_asym} and \ref{lem:eta_asym},
\begin{align*}
	\Log \left(\frac{1}{(q^2;q^2)_n}\right)&=\Log \left(\frac{(q^{2n+2};q^2)_\infty}{(q^2;q^2)_\infty} \right)
	\\
    &= \left(\frac{\pi^2}{6} - \Li_2 (\sqrt{2}-1)\right) \frac{1}{2z} -\left( \frac{u}{\sqrt{z}} + \frac{1}{2} \right) \log(2-\sqrt{2})
	+ \frac{1}{2} \Log\left(\frac{z}{\pi} \right) \\
    & \quad -\frac{u^2}{\sqrt{2}} - \frac{z}{12}
	+ \psi_{\sqrt{2}-1} \left(\frac{u}{\sqrt{z}}, 2z \right) + O\left(z^L\right)
\end{align*}
for all $L \in \mathbb{N}$.
Similarly, with letting $\nu:=\frac{u}{\sqrt{z}}-\frac32$ and using 
$$(-q^{-1};q^2)_\infty = (1+q^{-1}) \frac{(q^2;q^2)_\infty^2}{(q;q)_\infty (q^4;q^4)_\infty},$$
we get
\begin{align*}
	\Log (-q^{-1};q^2)_n &=\Log \left(\frac{(-q^{-1};q^2)_\infty}{(-q^{2n-3+2};q^2)_\infty} \right) \\
	& = \left( \Li_2 (1-\sqrt{2}) +\frac{\pi^2}{12}\right) \frac{1}{2z} + \left( \frac{u}{2\sqrt{z}} + \frac12 \right) \log(2) 
	  -\left(1-\frac1{\sqrt 2}\right)\! \left(\frac{u}{\sqrt{z}}-\frac32\right)^2\! z \\
      &
	\quad + \frac{11z}{24}- \psi_{1-\sqrt{2}} \left(\frac{u}{\sqrt{z}}-\frac32, 2z \right) -\sum_{m \geq 1} \frac{\Li_{1-2m}(-1)}{(2m)!} z^{2m} + O\left(z^L\right). 
\end{align*}
For $q^{n^2+n}$, we have
\[
	-(n^2 +n) z = - \frac{\log^2 (\sqrt{2}+1)}{4z} -\frac{\log(\sqrt{2}+1)}{2}  - \frac{\log (\sqrt{2}+1)}{\sqrt{z}}u -\sqrt{z}u  - u^2.
\]

In sum, using \cite[eq. (11)]{L}
\[
	 \Li_2 (\sqrt{2}-1) -\Li_2 (1-\sqrt{2}) = \frac{\pi^2}{8} - \frac{\log^2 (\sqrt{2}+1)}{2},
\]
we arrive at
\[
	\frac{ n q^{n^2+n} (-q^{-1};q^2)_n}{(q^2;q^2)_n}
	=\left( \frac{\log (\sqrt{2}+1)}{2z} + \frac{u}{\sqrt{z}}  \right)  2^{\frac14}\sqrt{\frac{z}{\pi}}  \exp\left(  \frac{\pi^2}{16z} -2u^2 + \mathcal{E} + O\left(z^L\right) \right),
\]
where 
\[
	\mathcal{E}=-\frac{3\sqrt{2}-4}{2}\sqrt{z}u -\frac{15-9\sqrt{2}}{8}z + \psi_{\sqrt{2}-1} \left(\frac{u}{\sqrt{z}}, 2z \right)
	- \psi_{1-\sqrt{2}} \left(\frac{u}{\sqrt{z}}-\frac32, 2z \right)
	-\sum_{m \geq 1} \frac{\Li_{1-2m}(-1)}{(2m)!} z^{2m}.
\]

As $y \ll \varepsilon^{\frac13+\delta}$, we have $z \asymp \varepsilon$. From $n \in \mathcal{N}_{\varepsilon}$, we have $u \ll \varepsilon^{-\frac16+\delta_1}$. Thus, from \eqref{eq:psi}, we find that $\mathcal{E} \ll \varepsilon^{3\delta_1}$ holds uniformly in $u$ since the constant in the error term is independent of $u$.
Therefore, we conclude that
\begin{align*}
	\frac{ n q^{n^2+n} (-q^{-1};q^2)_n}{(q^2;q^2)_n}
	&= \frac{\log(\sqrt{2}+1) }{2^{\frac34}\sqrt{\pi z}}\left( 1 + O\left(\varepsilon^{\frac13+\delta_1}\right) \right) 
	\exp\left(  \frac{\pi^2}{16z} -2u^2\right) \left(1+O\left(\varepsilon^{3\delta_1}\right)\right). \qedhere
\end{align*}
\end{proof}

Using the bijection
\[
	 n \mapsto u = - \frac{\log (\sqrt{2}+1)}{2\sqrt{z}} + n\sqrt{z},
\]
we can  derive the asymptotic of $H_{1,1}(q)$ for $y$ small. As the proof proceeds similarly to Theorem \ref{prop:S11_narrow}, we omit the proof here.

\begin{proposition}\label{prop:H11_narrow}
	If $y \ll \varepsilon^{\frac{2}{5}+\delta}$ for some $\delta>0$, then we have, as $z \to 0$, 
	\[
		\sum_{n \in \mathcal{N}_{\varepsilon}} \frac{ n q^{n^2+n} (-q^{-1};q^2)_n}{(q^2;q^2)_n} = \frac{\log(\sqrt{2}+1) }{2^{5/4}}  z^{-1} e^{\frac{\pi^2}{16 z}} + O\left(\varepsilon^{-\frac{9}{10}} e^{\frac{\pi^2}{16\varepsilon}}\right).
	\]
\end{proposition}

Now we give an error estimate for the sum of summands far from the peak.
\begin{proposition}\label{prop:G_far}
	For every $L \in \mathbb{N}$, as $\Re(z) \to 0$ in the right half-plane, 
	\[
		\left| \mathcal{E}_{1,1}(q) \right| = \left|\sum_{n \in \mathbb{N} \setminus \mathcal{N}_{\varepsilon}}  \frac{ n q^{n^2+n} (-q^{-1};q^2)_n}{(q^2;q^2)_n} \right| \ll \varepsilon^{L-1} e^{\frac{\pi^2}{16\varepsilon}}.
	\]
\end{proposition}
\begin{proof}
Since
\[
	\left|  \frac{ n q^{n^2+n} (-q^{-1};q^2)_n}{(q^2;q^2)_n}\right| \leq  \frac{ n |q|^{n^2+n} (-|q|^{-1}; |q|^2)_n}{(|q|^2;|q|^2)_n}, 
\]
it is sufficient to consider the case when $z$ is real.
Thus, suppose that $z=\varepsilon$ is real, i.e., $y=0$. Fix $n \in \mathbb{N} \setminus \mathcal{N}_{\varepsilon}$, and let $\mu \in \mathbb{R}$ be such that $n = \frac{\log(\sqrt{2}+1)}{2\varepsilon}+\mu$.

First, we consider the sum over $n$ with $\mu>0$. Then $q^{2n}=(\sqrt{2}-1)e^{-2\mu \varepsilon}$, and
\[
	\log\left((q^{2n+2};q^2)_\infty \right) = \sum_{m \geq 1} \log\left(1-q^{2n+2m}\right) = -\sum_{m \geq 1}\sum_{k \geq 1}\frac{q^{k(2n+2m)}}{k} 
	= -\sum_{k \geq 1} \frac{(\sqrt{2}-1)^k}{k}\frac{e^{-2k\mu\varepsilon}}{e^{2k\varepsilon}-1}.
\]
As $e^x> 1+x$ for all $x \neq 0$ and $\frac{1}{e^x-1}>\frac{1}{x}-\frac{1}{2}$ for $x>0$, we have
\[
	\frac{e^{-2k\mu\varepsilon}}{e^{2k\varepsilon}-1} > (1-2k\mu\varepsilon)\left(\frac{1}{2k\varepsilon} -\frac12 \right) = \frac{1}{2k\varepsilon} - \left( \mu +\frac12 \right) + k\mu \varepsilon,
\]
which implies
\[
	\log\left((q^{2n+2};q^2)_\infty \right)  < -\frac{\Li_2(\sqrt{2}-1)}{2\varepsilon} - \log(2-\sqrt{2})\left( \mu +\frac12 \right)-\frac{\mu\varepsilon}{\sqrt{2}} .
\]
Together with Lemma \ref{lem:eta_asym}, we find
\begin{align}
	 \log\left(\frac{1}{(q^2;q^2)_n}\right)& =\log\left(\frac{(q^{2n+2};q^2)_\infty}{(q^2;q^2)_\infty} \right) \notag\\
    & <  \left(\frac{\pi^2}{6}-\Li_2(\sqrt{2}-1)\right)\frac{1}{2\varepsilon} + \frac{1}{2} \log\left(\frac{\varepsilon}{\pi(2-\sqrt{2})} \right) \notag \\ & \quad - \frac{\varepsilon}{12} -\log(2-\sqrt{2}) \mu -\frac{\mu\varepsilon}{\sqrt{2}}+ O\left(\varepsilon^L\right). \label{pro6.3_1}
\end{align}
We also have
\[
	- \log\left((-q^{2n-3+2};q^2)_\infty\right)% = - \sum_{m \geq 1} \log(1+q^{2n-3+2m}) =  \sum_{m \geq 1} \sum_{k \geq 1} \frac{(-1)^k q^{k(2n-3+2m)}}{k} 
	= \sum_{k \geq 1} \frac{(-1)^{k} (\sqrt{2}-1)^k}{k} \frac{e^{-2k\mu \varepsilon +3k\varepsilon}}{e^{2k\varepsilon}-1} =:\Sigma_1 + \Sigma_2,
\]
where
\[
	\Sigma_1:= - \sum_{\substack{k \geq 1\\k: \text{odd}}} \frac{(\sqrt{2}-1)^k}{k} \frac{e^{-2k\mu \varepsilon +3k\varepsilon}}{e^{2k\varepsilon}-1} \qquad\text{and}\qquad  
	\Sigma_2:= \sum_{\substack{k \geq 1\\k: \text{even}}} \frac{(\sqrt{2}-1)^k}{k} \frac{e^{-2k\mu \varepsilon +3k\varepsilon}}{e^{2k\varepsilon}-1}.
\]
For the sum over odd $k$, we obtain
\begin{align*}
	\Sigma_1
	 %&- \sum_{\substack{k \geq 1\\k: \text{odd}}} \frac{(\sqrt{2}-1)^k}{k} \frac{e^{-2k\mu \varepsilon +3k\varepsilon}}{e^{2k\varepsilon}-1}  \\
	 &< - \frac1{2\varepsilon} \sum_{\substack{k \geq 1\\k: \text{odd}}} \frac{(\sqrt{2}-1)^k}{k^2} - (1-\mu) \sum_{\substack{k \geq 1\\k: \text{odd}}} \frac{(\sqrt{2}-1)^k}{k} + \varepsilon \left( \frac32-\mu \right) \sum_{\substack{k \geq 1\\k: \text{odd}}} (\sqrt{2}-1)^k\\
	 &= - \frac1{2\varepsilon} \sum_{\substack{k \geq 1\\k: \text{odd}}} \frac{(\sqrt{2}-1)^k}{k^2} + \frac{1}{2}(1-\mu) \log(\sqrt{2}-1) + \frac{\varepsilon}{2} \left( \frac32-\mu \right).
\end{align*}
For small $\varepsilon$, we have $\mu >\frac{3}{2}$. Thus, we use $e^{k(-2\mu+3) \varepsilon} < 1$ and $e^{2k\varepsilon}-1>2k\varepsilon$ to get
\[
	%\sum_{\substack{k \geq 1\\k: \text{even}}} \frac{(\sqrt{2}-1)^k}{k} \frac{e^{-2k\mu \varepsilon +3k\varepsilon}}{e^{2k\varepsilon}-1} 
	%< \sum_{\substack{k \geq 1\\k: \text{even}}} \frac{(\sqrt{2}-1)^k}{k} \frac{e^{-2k\mu \varepsilon +3k\varepsilon}}{2k\varepsilon} 
	\Sigma_2 < \frac1{2\varepsilon}\sum_{\substack{k \geq 1\\k: \text{even}}} \frac{(\sqrt{2}-1)^k}{k^2}.
\]
Hence,
\begin{align}
	\log\big( (-q^{-1};q^2)_n\big)& =\log \left(\frac{(-q^{-1};q^2)_\infty}{(-q^{2n-3+2};q^2)_\infty} \right) \notag \\
    &< \frac{\Li_2(1-\sqrt{2})}{2\varepsilon} + \frac{1}{2}(1-\mu) \log(\sqrt{2}-1) + \frac{\varepsilon}{2} \left( \frac32-\mu \right) \notag\\
	 &\quad + \frac{\pi^{2}}{24\varepsilon} + \log(2) + \frac{11\varepsilon}{24} -\sum_{m \geq 1} \frac{\Li_{1-2m}(-1)}{(2m)!} \varepsilon^{2m}+ O\left(\varepsilon^L\right). \label{pro6.3_2}
\end{align}
For $q^{n^2+n}$, we have
\begin{equation}
	-(n^2+n) \varepsilon = - \frac{\log^2 (\sqrt{2}+1)}{4\varepsilon}  - \log (\sqrt{2}+1) \left( \mu+\frac12\right)  - \mu^2 \varepsilon   - \mu \varepsilon. \label{pro6.3_3}
\end{equation}
Therefore, putting \eqref{pro6.3_1}, \eqref{pro6.3_2} and \eqref{pro6.3_3} together, we find that for $\mu>0$,
\begin{multline*}
	\log\left(\frac{q^{n^2+n} (-q^{-1};q^2)_n}{(q^2;q^2)_n} \right) <%&  \frac{\pi^2}{16\varepsilon}  + \frac{1}{2} \log\left(\frac{\varepsilon}{\pi(2-\sqrt{2})} \right) - \frac{\varepsilon}{12} -\log(2-\sqrt{2}) \mu -\frac{\mu\varepsilon}{\sqrt{2}}\\
	%& + \frac{1}{2}(1-\mu) \log(\sqrt{2}-1) + \frac{\varepsilon}{2} \left( \frac32-\mu \right)+ \log(2) + \frac{11\varepsilon}{24} \\
	%& -\sum_{m \geq 1} \frac{\Li_{1-2m}(-1)}{(2m)!} \varepsilon^{2m}- \log (\sqrt{2}+1) \left( \mu+\frac12\right)  - \mu^2 \varepsilon   - \mu \varepsilon+ O\left(\varepsilon^L\right) \\
	 \frac{\pi^2}{16\varepsilon} -\frac{1}{2}\log(2\sqrt{2}-2)\mu - \frac{1}{2} \log\left(\frac{\pi} {\varepsilon(4-2\sqrt{2})}\right)  - \left(\frac{3+\sqrt{2}}{2}\mu-\frac{9}{8}\right)\varepsilon \\
	 - \mu^2 \varepsilon  -\sum_{m \geq 1} \frac{\Li_{1-2m}(-1)}{(2m)!} \varepsilon^{2m} + O\left(\varepsilon^L\right). 
\end{multline*}
If $\delta>0$, $\ell< -\frac12-\delta$, and $|\mu|>\varepsilon^\ell$, then there exists $\varepsilon_0>0$, independent of $\ell$, such that for $\varepsilon < \varepsilon_0$,
\begin{align*}
	\log\left(\frac{ q^{n^2+n} (-q^{-1};q^2)_n}{(q^2;q^2)_n} \right) <&  \frac{\pi^2}{16\varepsilon} -\frac12 \varepsilon^{2\ell+1}.
\end{align*}
As in the proof of Proposition \ref{prop:far}, we obtain
\[
	\sum_{n > \frac{\log(\sqrt{2}+1) }{2\varepsilon} + \varepsilon^{-\frac35}}  \frac{n q^{n^2+n} (-q^{-1};q^2)_n}{(q^2;q^2)_n}  \ll \varepsilon^{L-1} e^{\frac{\pi^2}{16\varepsilon}}
\]
for all $L \in \mathbb{N}$ as $\varepsilon \to 0$.

Next, we consider the sum over $n$ with $\mu<0$. We let $a_n := \frac{ n q^{n^2+n} (-q^{-1};q^2)_n}{(q^2;q^2)_n}$. Then the peak occurs at $q^{2n} \approx \sqrt{2}-1$ and
\[
	\frac{a_{n+1}}{a_{n}} 
	%&:= \frac{ (n+1) q^{(n+1)^2+(n+1)} (-q^{-1};q^2)_{n+1}}{(q^2;q^2)_{n+1}} \cdot \frac{(q^2;q^2)_{n}} { n q^{n^2+n} (-q^{-1};q^2)_{n}}\\
	= \frac{n+1}{n} \cdot \frac{q^{2n+2} (1+ q^{2n-1})}{1- q^{2n+2} } 
	= \frac{n+1}{n} \cdot \frac{1+ q^{-1} \exp(- 2n\varepsilon )}{ q^{-2} \exp (2n\varepsilon) -1 } \approx
	\frac{1+ \exp(- 2n\varepsilon ) }{ \exp (2n\varepsilon) -1 },
\]	
which is decreasing as $n$ grows for a fixed $\varepsilon >0$. Thus, for
\[
	n < \frac{\log(\sqrt{2}+1)}{2\varepsilon} - \varepsilon^{-\frac{3}{5}},
\]
we can find $\rho \in (0,1)$ so that 
\[
	\frac{a_n}{a_{n-1}} > \frac{1 + (\sqrt{2}-1) \exp(-2 \varepsilon^{\frac25} )}{ (\sqrt{2}+1) \exp(2\varepsilon^{\frac25})-1 } \geq \frac{1}{\rho}.
\]
Therefore, letting the value of $a_n$ at the peak be $M$, we obtain that for all $L \in \mathbb{N}$ as $\varepsilon \to 0$,
\[
	\sum_{n < \frac{\log(\sqrt{2}+1) }{2\varepsilon} - \varepsilon^{-\frac35}} a_n 
	\le \sum_{n < \frac{\log(\sqrt{2}+1) }{2\varepsilon} - \varepsilon^{-\frac35}} \rho^{\varepsilon^{-\frac35} } \rho^n M \ll  \varepsilon^{L} e^{\frac{\pi^2}{16\varepsilon}},
\]
where we use $M \ll \varepsilon^{-\frac12} e^{\frac{\pi^2}{16\varepsilon}}$ from Proposition \ref{prop:H_narrow}.
\end{proof}

%%%%%%%%%%%%%%%%%%%%%%%%%%%%%%%%%%%%%%%%%%%%%
\subsection{Asymptotic of $H_{1,1}(q)$}\label{sec:H_asymp}

The following proposition gives wider range estimates near the peak and here we omit the proof.

\begin{proposition}\label{prop:H_wider}
	Let $n \in \mathcal{N}_{\varepsilon}$ and $v \in \mathbb{C}$ be such that
	\begin{equation*}
		n = \frac{\log(\sqrt{2}+1)}{2\varepsilon} + \frac{v}{\sqrt{z}}. 
	\end{equation*}
	Suppose that $y \ll 1$. Then we have
	\begin{align*}
		\frac{ n q^{n^2+n} (-q^{-1};q^2)_n}{(q^2;q^2)_n}
		=  \frac{\log(w^{-1})}{(1+w)}\sqrt{\frac{w}{\pi z(1-w)}} \exp\left( \frac{\Lambda(y)}{2z} + \Log\left(\frac{w(1+w)}{1-w}\right) \frac{v}{\sqrt{z}} -\left(1+\frac{2w}{1-w^2}\right) v^2 \right)\\
		+ O\left( \varepsilon^{\frac15} \right) \frac{1}{\sqrt{z}}\exp\left( \frac{\Lambda(y)}{2z} + \Log\left(\frac{w(1+w)}{1-w}\right) \frac{v}{\sqrt{z}} -\left(1+\frac{2w}{1-w^2}\right) v^2 \right),
	\end{align*}
	where $w := (\sqrt{2}-1)^{1+iy}$ and 
	\[
		\Lambda(y) := \frac{\pi^2}{4} - \Li_2(w)+ \Li_2(-w)  - \frac12 (1+iy)^2 \log^2(\sqrt{2}+1).
	\]
\end{proposition}

Next, we show that the asymptotic of $H_{1,1}(q)$ in Proposition \ref{prop:H11_narrow} is valid for $|y|\ll 1$. To this end, we need to prove the following lemma, which can be proven in a similar way to Lemma \ref{lem:s}. 

\begin{lemma}\label{lem:H_s}
	Let $s(y):=\Re\left(\frac{\Lambda(y)}{2(1+iy)} - \frac{\pi^2}{16}\right)$. Then we have the following.
	\begin{enumerate}
		\item $s(y) \leq 0$ for all $y \in \mathbb{R}$, and the equality holds if and only if $y=0$.
		\item As $y \to 0$,
		\[
			s(y) = \left(-\frac{\pi^2}{16} + \frac12\log^2(\sqrt{2}+1) \right)y^2 + O\left(y^4\right).
		\]
	\end{enumerate}
\end{lemma}

Now, we are ready to prove the asymptotic of $H_{1,1}(q)$ as $z=\varepsilon(1+iy) \to 0$ in the bounded cone $|y|\ll 1$. 
\begin{proposition}\label{prop:H11_asymp}
	If $y \ll 1$, then we have, as $z \to 0$, 
	\[
		H_{1,1}(q) = \frac{\log(\sqrt{2}+1) }{2^{5/4}}  z^{-1} e^{\frac{\pi^2}{16 z}} + O\left(\varepsilon^{-\frac{9}{10}} e^{\frac{\pi^2}{16\varepsilon}}\right).
	\]
\end{proposition}
\begin{proof}
By Proposition \ref{prop:G_far},  the desired asymptotic follows from the claim:
\[
	z e^{-\frac{\pi^2}{16\varepsilon}} \left(\sum_{n \in \mathcal{N}_{\varepsilon}} \frac{ n q^{n^2+n} (-q^{-1};q^2)_n}{(q^2;q^2)_n} - \frac{\log(\sqrt{2}+1) }{2^{5/4}}  z^{-1} e^{\frac{\pi^2}{16 z}} \right) \ll \varepsilon^{\frac{1}{10}} .
\]

From Proposition~\ref{prop:H11_narrow}, it suffices to show that if $\varepsilon^{\frac12-\delta}\ll y \ll 1$ for some $\delta>0$, then we have, for all $L \in \mathbb{N}$,
\begin{equation}\label{eq:H_first_part}
	z e^{-\frac{\pi^2}{16\varepsilon}}\sum_{n \in \mathcal{N}_{\varepsilon}} \frac{ n q^{n^2+n} (-q^{-1};q^2)_n}{(q^2;q^2)_n} = O\left(\varepsilon^L\right).
\end{equation}
If $y \ll 1$, by Proposition \ref{prop:H_wider},
\begin{align*}
	z e^{-\frac{\pi^2}{16\varepsilon}}\sum_{n \in \mathcal{N}_{\varepsilon}}  \frac{ n q^{n^2+n} (-q^{-1};q^2)_n}{(q^2;q^2)_n} 
	= \frac{\log(w^{-1})}{1+w} \sqrt{\frac{w z}{\pi(1-w)}} e^{\frac{1}{2\varepsilon}\left(\frac{\Lambda(y)}{1+iy}-\frac{\pi^2}{8}\right)} \sum_{n \in \mathcal{N}_{\varepsilon}} e^{-\left(1+\frac{2w}{1-w^2}\right)v^2 +\frac{v}{\sqrt{z}} \Log\left(\frac{w(1+w)}{1-w}\right)}\\
    + O\left(\varepsilon^{\frac15}\right) \sqrt{z} e^{\frac{1}{2\varepsilon}\left(\frac{\Lambda(y)}{1+iy}-\frac{\pi^2}{8}\right)} \sum_{n \in \mathcal{N}_{\varepsilon}}  e^{-\left(1+\frac{2w}{1-w^2}\right)v^2 +\frac{v}{\sqrt{z}} \Log\left(\frac{w(1+w)}{1-w}\right)},
\end{align*}
where $w = (\sqrt{2}-1)^{1+iy}$.
Thus, to prove \eqref{eq:H_first_part}, it is sufficient to show that the exponent
\begin{equation}\label{eq:H_exp}
	\frac{1}{2\varepsilon}\left(\frac{\Lambda(y)}{1+iy}-\frac{\pi^2}{8}\right) -\left(1+\frac{2w}{1-w^2}\right)v^2 +\frac{v}{\sqrt{z}} \Log\left(\frac{w(1+w)}{1-w}\right)
\end{equation}
has negative real part of size $\gg \varepsilon^{-\delta_0}$ for some $\delta_0>0$ since the other terms are bounded as $z\to 0$.

If $\varepsilon^{\frac12-\delta} \ll y \ll 1$ for some $\delta>0$, then %$ \delta < \frac12$.
the real part of $\frac{1}{2\varepsilon}\left(\frac{\Lambda(y)}{1+iy}-\frac{\pi^2}{8}\right)$ is $\frac{s(y)}{\varepsilon}$, which is negative by Lemma \ref{lem:H_s} (1). We also obtain, as $\varepsilon \to 0$, 
\[
	\frac{s(y)}{\varepsilon} \gg \frac{y^2}{\varepsilon} \gg \varepsilon^{-2\delta}
\]
by Lemma \ref{lem:H_s} (2). 
From the Taylor series expansion, as $y \to 0$, 
\[
	\Re\left( \Log\left(\frac{w(1+w)}{1-w}\right) \right) = \log \left|\frac{w(1+w)}{1-w} \right| = -\frac{\log^2(\sqrt{2}-1)}{\sqrt{2}} y^2 + O\left(y^4\right) \ll y^2.
\]
Since $|\nu|=\left|\frac{v}{\sqrt{z}}\right| \ll \varepsilon^{-\frac35}$, 
\[
	\Re\left( \frac{v}{\sqrt{z}} \Log\left(\frac{w(1+w)}{1-w}\right) \right) \ll \varepsilon^{-\frac35} y^2.
\]	
As $\frac{y^2}{\varepsilon} \gg \varepsilon^{-\frac35} y^2$, $\frac{s(y)}{y}$ dominates $\Re\left( \frac{v}{\sqrt{z}} \Log\left(\frac{w(1+w)}{1-w}\right) \right)$. For the term $\left(1+\frac{2w}{1-w^2}\right)v^2$, we consider the following two cases.
\begin{enumerate}
	\item Suppose $\varepsilon^{\frac12-\delta} \ll y \ll \varepsilon^{\frac14}$. As $v$ is a real multiple of $\sqrt{z}$ and $y\ll \varepsilon^{\frac14}$, we have that for $y \ll \varepsilon^{\frac14}$,
	\[
		\Re \left(\left(1+\frac{2w}{1-w^2}\right) (1+iy)\right) >0.
	\]
	Thus,  $-\left(1+\frac{2w}{1-w^2}\right)v^2$ has negative real part as $z\to 0$.
	
	\item If $\varepsilon^{\frac13} \ll y \ll 1$, then $|z|\ll \varepsilon$. As $n \in \mathcal{N}_{\varepsilon}$, we have $|v|\ll \varepsilon^{-\frac1{10}}$. Thus, with the fact  $\left|1+\frac{2w}{1-w^2}\right| \leq 2$, we find $-\left(1+\frac{2w}{1-w^2}\right)v^2 \ll  \varepsilon^{-\frac15}$. 
\end{enumerate}
In either case, the exponent $\frac{s(y)}{\varepsilon}$ has size $\gg \varepsilon^{-2\delta}$ and dominates the other exponents in \eqref{eq:H_exp} that have positive real part. Hence, we prove the claim \eqref{eq:H_first_part} by taking $\delta_0=2\delta$.

Next, in the expression  $\frac{\log(\sqrt{2}+1) }{2^{5/4}} e^{\frac{\pi^2}{16 z}-\frac{\pi^2}{16\varepsilon}}$, the exponential part has negative real part of size
\[
	\Re\left(\frac{\pi^2}{16 z}-\frac{\pi^2}{16\varepsilon}\right) = \frac{\pi^2}{16\varepsilon} \left( \Re\left(\frac{1}{1+iy}\right) -1\right) = - \frac{\pi^2}{16\varepsilon} \frac{y^2}{1+y^2} \gg \varepsilon^{-2\delta}.
\]
for $\varepsilon^{\frac12-\delta} \ll y \ll 1$. It follows that for $\varepsilon^{\frac12-\delta} \ll y \ll 1$,  %we have, for all $L \in \mathbb{N}$,
\[
	\frac{\log(\sqrt{2}+1) }{2^{5/4}} e^{\frac{\pi^2}{16 z}-\frac{\pi^2}{16\varepsilon}} \ll \varepsilon^L \quad \text{for all $L \in \mathbb{N}$}. \qedhere
\]
%Finally, for $y \ll \varepsilon^{\frac25+\delta}$ for some $\delta>0$, then, by Proposition \ref{prop:H11_narrow}, we have
%\[
%	\sum_{n \in \mathcal{N}_{\varepsilon}} \frac{ n q^{n^2+n} (-q^{-1};q^2)_n}{(q^2;q^2)_n} - \frac{\log(\sqrt{2}+1) }{2^{5/4}}  z^{-1} e^{\frac{\pi^2}{16 z}} \ll \varepsilon^{-\frac{9}{10}} e^{\frac{\pi^2}{16\varepsilon}}.
%\]
%Choosing $\delta>0$ sufficiently small, the two cases, $y \ll \varepsilon^{\frac25+\delta}$ and $\varepsilon^{\frac12-\delta} \ll y \ll 1$, cover the full range $y \ll 1$, which complete the proof.
\end{proof}

%%%%%%%%%%%%%%%%%%%%%%%%%%%%%%%%%%%%%%%%%%%%% 
\subsection{Proofs of Theorems \ref{thm:g1_asymp} and \ref{thm:g2_asymp}}
In this subsection, we finally prove asymptotic formulas. 
\begin{proof}[Proof of Theorem \ref{thm:g1_asymp}]
By Proposition \ref{prop:H11_asymp}, we have, as $z=\varepsilon(1+iy) \to 0$, where $\varepsilon>0$ and $y \ll 1$,  
\[
	H_{1,1}(q) \sim \frac{\log(\sqrt{2}+1) }{2^{\frac54}}  z^{-1} e^{\frac{\pi^2}{16 z}}.
\]
Similarly to Proposition \ref{prop:H11_asymp}, we can show that $H_{1,2}(q)$ has the same asymptotic as $H_{1,1}(q)$.

Since $g_{1,1}(n)$ and $g_{2,2}(n)$ are clearly weakly increasing, we can employ Theorem  \ref{Tauberian} to obtain the asymptotics of $g_{2,1}(n)$ and $g_{1,2}(n)$ as $n \to \infty$. 
\end{proof}

Now we turn to prove Theorem~\ref{thm:g2_asymp}.

\begin{proof}[Proof of Theorem \ref{thm:g2_asymp}]
Recall Propositions \ref{gen_H1} and \ref{gen_H2}:
\begin{align*}
	H_{2,1}(q)&=\sum_{n\ge 0} g_{2,1}(n)q^n = \frac{1}{(q, q^5, q^6;q^8)_{\infty}} \frac{q+q^5+q^6}{1-q^{8}},\\
	H_{2,2}(q)&=\sum_{n\ge 0} g_{2,2}(n) q^{n} = \frac{1}{(q,q^5,q^6;q^8)_\infty}\left( \frac{q^5+q^6+q^9}{1-q^8} + \frac{q^2+q^{10}-q^{11}+q^{12}}{1-q^{16}} \right).
\end{align*}
Similarly to Proposition \ref{prop:H11_asymp}, we can find that as $z=\varepsilon(1+iy) \to 0$,   
\[
	\frac{1}{(q, q^5, q^6;q^8)_{\infty}} = \sum_{n \geq 0} \frac{ q^{n^2+n} (-q^{-1};q^2)_n}{(q^2;q^2)_n} \sim \frac{1}{2^{\frac14}} e^{\frac{\pi^2}{16 z}},
\]
where $\varepsilon>0$ and $y \ll 1$. Also, using the asymptotics
\[
	\frac{q+q^5+q^6}{1-q^8} \sim \frac{3}{8z}  \qquad\text{and}\qquad 
	\frac{q^5+q^6+q^9}{1-q^8} + \frac{q^2+q^{10}-q^{11}+q^{12}}{1-q^{16}} \sim \frac{1}{2z} \qquad \text{as } z \to 0,
\]
we have, as $z=\varepsilon(1+iy) \to 0$, where $\varepsilon>0$ and $y \ll 1$,  
\[
	H_{2,1}(q) \sim \frac{3}{2^{\frac{13}{4}}z} e^{\frac{\pi^2}{16z}} \qquad\text{and}\qquad 
	H_{2,2}(q) \sim \frac{1}{2^{\frac{5}{4}}z} e^{\frac{\pi^2}{16z}}. 
\]

Since $g_{2,1}(n)$ and $g_{2,2}(n)$ are clearly weakly increasing, we can employ Theorem  \ref{Tauberian} to obtain the asymptotics of $g_{2,1}(n)$ and $g_{2,2}(n)$ as $n \to \infty$. 
\end{proof} 

%%%%%%%%%%%%%%%%%%%%%%%%%%%%%%%%%%%%%%%%%%%%%
%          Section.   %
%%%%%%%%%%%%%%%%%%%%%%%%%%%%%%%%%%%%%%%%%%%%%
\section{Final Remarks}\label{sec7}

In this paper, we investigate the $t$-hook inequalities between two equinuermous sets arising from the first Rogers--Ramanujan identity and the first little G\"ollnitz identity. This analysis can be extended to the second Rogers--Ramanujan identity and the second little G\"ollnitz identity. The differences between the first and second identities are the following. For part difference conditions, parts are all greater than $1$, while for part congruence conditions, parts are congruent to $2, 3$ modulo $5$ for the second Rogers--Ramanujan, and congruent to $2, 3, 7$ modulo 8 for the second little G\"ollnitz. Another example with gap $2$ conditions is the G\"ollnitz--Gordon identities \cite{gollnitz, gordon}.  The first G\"ollnitz--Gordon identity states that the number of partitions of $n$ with parts differing by $2$ and no even parts differing by exactly $2$ is equal to the number of partitions of $n$ into parts congruent to $1, 4$ or $7$ modulo $8$. As the corresponding generating functions for these partition sets with  the gap conditions are likely Nahm sums, we hope that the asymptotic method used in this paper will guide the further analysis of such generating functions. For example, the number of $1$-hooks in partitions with the gap 2 condition from the second Rogers--Ramanujan identity is generated by $\sum_{n \ge 1} \frac{n q^{n^2+n}}{(q;q)_{n}}$, to which our method can be applied.

%For $t=1,2$, it is {\color{cyan} more straightforward to find the generating functions for the number of $t$-hooks in these partitions compared to the partitions we considered in our paper}.

 %5{\color{cyan} We anticipate similar results hold true.}

One of the key players in our proofs was the generating function for the number of $t$-hooks. For large $t$, it seems not easy to find the generating function. Based on numerical experiments for small $t$ and $n$, we have observed that for sufficiently large $n$, there are more $1$-hooks in all partitions of $n$ satisfying certain gap conditions than in all partitions of $n$ satisfying the corresponding congruence conditions while the inequality reverses for $t>1$, and the congruence condition partitions have more $t$-hooks. In particular, we conjecture that for $t>2$,
\[
    r_{1,t} (n) < r_{2,t} (n) \quad\text{and}\quad g_{1,t}(n) < g_{2,t} (n)
\]
holds for all sufficiently large integers $n$.

%%%%%%%%%%%%%%%%%%%%%%%%%%%%%%%%%%%%%%%%%%%%%
%          Section.     %  
%%%%%%%%%%%%%%%%%%%%%%%%%%%%%%%%%%%%%%%%%%%%%
\section*{Acknowledgments}

Aritram Dhar would like to thank George E. Andrews for hosting him from May 7 - June 25, 2025 at the Pennsylvania State University as a visiting research scholar where this project started after discussions with Ae Ja Yee. Byungchan Kim was supported by the Basic Science Research Program through the National Research Foundation of Korea (NRF) funded by the Ministry of Science and ICT (RS-2025-16065347).
Eunmi Kim was supported by the Basic Science Research Program through the National Research Foundation of Korea (NRF) funded by the Ministry of Education (RS-2023-00244423). 
Ae Ja Yee was partially supported by a grant ($\#$633963) from the Simons Foundation.

%%%%%%%%%%%%%%%%%%%%%%%%%%%%%%%%%%%%%%%%%%%%%
%%% References 
%%%%%%%%%%%%%%%%%%%%%%%%%%%%%%%%%%%%%%%%%%%%%


\begin{thebibliography}{99}

\bibitem{andrews1969}
G.~E.~Andrews, {\it Two theorems of Euler and a general partition theorem}, Proc.~Amer.~Math.~Soc. {\bf 20} (1969), 499--502.

\bibitem{andrews_book}
G. E. Andrews, The Theory of Partitions, Addison-Wesley, Reading, 1976 (Reprinted: Cambridge University Press, Cambridge, 1985).

\bibitem{andrews_siam}
G. E. Andrews, {\it Applications of basic hypergeometric functions}, SIAM Rev. {\bf 16} (1974), 441--484. 

\bibitem{andrews2017}
G. E. Andrews, {\it Euler's partition identity and two problems of George Beck},  Math. Student {\bf 86} (2017), 115--119. 

\bibitem{AKY}
G. E. Andrews, R. Kumar, and A. J. Yee, {\it On Euler's partition theorem}, Frontiers in Combinatorics and Number Theory
Vol. 1, 2026, pp. 26--30. doi:10.3934/fcnt.2026003

\bibitem{ballantine1}
C. Ballantine, H. Burson, W. Craig, A. Folsom, and B. Wen, {\it Hook length biases and general linear partition inequalities}, Res. Math. Sci. {\bf 10} (2023), Paper No. 41, 36 pp.

\bibitem{BJM}
K.~Bringmann, C.~Jennings-Shaffer, and K.~ Mahlburg, {\it On a Tauberian theorem of Ingham and Euler--Maclaurin summation}, Ramanujan J. {\bf 61} (2023), 55--86.

\bibitem{BMRS}
K. Bringmann,  S. H. Man, L. Rolen,  and M. Storzer, {\it Asymptotics of parity biases for partitions into distinct parts via Nahm sums}, Proc. Lond. Math. Soc. (3) {\bf 129} (2024), no. 6, Paper No. e70010, 40 pp.

\bibitem{craig1}
W. Craig, M. L. Dawsey, and G.-N. Han, {\it Inequalities and asymptotics for hook numbers in restricted partitions}, preprint arXiv:2311.15013v1. 

%\bibitem{folsom}
%A. Folsom, {\it Rogers--Ramanujan moment bias}, Ramanujan J. {\bf 68} (2025), Paper No. 25, 14 pp. 

\bibitem{GZ}
S. Garoufalidis and D. Zagier, {\it Asymptotics of Nahm sums at roots of unity}, Ramanujan J. {\bf 55} (2021), 219--238. 

\bibitem{glaisher}
J. W. L. Glaisher, {\it A theorem in partitions}, Messenger of Math. {\bf 12} (1883), 158--170.

\bibitem{gollnitz}
H. G\"    , {\it Partitionen mit differenzenbedingungen}, J. Reine Angew. Math. {\bf 225} (1967), 154--190.

\bibitem{gordon}
B. Gordon, {\it Some continued fractions of the Rogers--Ramanujan type}, Duke Math J. {\bf 32} (1965), 741--748.

\bibitem{K}
M.~Katsurada, {\it Asymptotic expansions of certain q-series and a formula of Ramanujan for specific values of the Riemann zeta function}, Acta Arith. {\bf 107} (2003), 269--298.

\bibitem{L}
F. M. S. Lima, {\it New definite integrals and a two-term dilogarithm identity} Indag. Math. {\bf 23} (2012), 1--9.

\bibitem{ramanujan}
S. Ramanujan, {\it Proof of certain identities in combinatory analysis}, Proc. Cambridge Philos. Soc. {\bf 19} (1919), 214--216.

\bibitem{rogers}
L. J. Rogers, {\it Second memoir on the expansion of certain infinite products}, Proc. London Math. Soc. {\bf 25} (1894), 318--343.

\bibitem{schur}
I. J. Schur, {\it Zur additiven Zahlentheorie}, S.-B. Akad. Wiss. Berlin, (1926) 488--495.

\bibitem{sylvester}
J. J. Sylvester, {\it A constructive theory of partitions, arranged in three acts, an interact and an exodion}, Amer. J. Math., 5 (1882), 251-330 and 6 (1884), 334-336 (or pp. 1--83 of The Collected Mathematical Papers of James Joseph Sylvester, Vol. 4, Cambridge University Press, Cambridge,
1912).

\bibitem{beck1}
The On-Line Encyclopedia of Integer Sequences, Sequences A090867 and A265251, https://oeis.org.

%\bibitem{beck2}
%The On-Line Encyclopedia of Integer Sequences, Sequence ,  https://oeis.org.

\bibitem{VZ}
M.~ Vlasenko, S.~Zwegers, {\it Nahm's conjecture: asymptotic computations and counterexamples}, Commun. Number Theory Phys. {\bf 5} (2011), 617--642.

\bibitem{W}
L.~Wang, {\it Explicit forms and proofs of Zagier's rank three examples for Nahm's problem}, Adv. Math. {\bf 450} (2024), Paper No. 109743.

\bibitem{Z} 
D. Zagier, {\it The dilogarithm function}, In: Frontiers in Number Theory, Physics, and Geometry II, P. Cartier, B. Julia, P. Monsia, P. Vanhove (eds), Springer-Verlag, Berlin-Heidelberg-New York (2006), 3--65.

\end{thebibliography}
\end{document}